\documentclass[%
 aip,
 amsmath,amssymb,
 reprint,%
]{revtex4-1}

\usepackage{graphicx}
\usepackage{dcolumn}
\usepackage{bm}

\usepackage{hyperref}
\usepackage{url}
\usepackage{etoolbox}

\usepackage[utf8]{inputenc}
\usepackage[T1]{fontenc}
\usepackage{mathptmx}

\usepackage[labelformat=simple
]{subcaption}
\usepackage{lipsum}
\usepackage{amsthm}
\usepackage{filecontents}
\usepackage{booktabs}

\newtheorem{theorem}{Theorem}
\newtheorem{lemma}[theorem]{Lemma}

\begin{document}

\preprint{AIP/123-QED}

\title[Entanglement-Embedded Recurrent Network Architecture]{Entanglement-Embedded Recurrent Network Architecture: \\ Tensorized Latent State Propagation and Chaos Forecasting}

\author{Xiangyi Meng}
 \email{xm@bu.edu.}
\affiliation{Center for Polymer Studies and Department of Physics, Boston University, USA}%

\author{Tong Yang}
 \email{yangto@bc.edu.}
\affiliation{Department of Physics, Boston College, USA}%

\date{\today}

\begin{abstract}
	Chaotic time series forecasting has been far less understood despite its tremendous potential in theory and real-world applications. Traditional statistical/ML methods are inefficient to capture chaos in nonlinear dynamical systems, especially when the time difference $\Delta t$ between consecutive steps is so large that a trivial, ergodic local minimum would most likely be reached instead. Here, we introduce a new long-short-term-memory (LSTM)-based recurrent architecture by tensorizing the cell-state-to-state propagation therein, keeping the long-term memory feature of LSTM while simultaneously enhancing the learning of short-term nonlinear complexity. We stress that the global minima of chaos can be most efficiently reached by tensorization where all nonlinear terms, up to some polynomial order, are treated explicitly and weighted equally. The efficiency and generality of our architecture are systematically tested and confirmed by theoretical analysis and experimental results. In our design, we have explicitly used two different many-body entanglement structures---matrix product states (MPS) and the multiscale entanglement renormalization ansatz (MERA)---as physics-inspired tensor decomposition techniques, from which we find that MERA generally performs better than MPS, hence conjecturing that the learnability of chaos is determined not only by the number of free parameters but also the tensor complexity---recognized as how entanglement entropy scales with varying matricization of the tensor.
\end{abstract}

\maketitle


\section{Introduction}

	Time series forecasting~\cite{forecast_msa18}, despite its undoubtedly tremendous potential in both theoretical issues (e.g., mechanical analysis, ergodicity) and real-world applications (e.g., traffic, weather), has long been known as an intricate field. From classical work of statistics such as auto-regressive moving average (ARMA) families~\cite{time-ser-anal_bjrl15} and basic hidden Markov models (HMM)~\cite{hidden-markov-model_rj86} to contemporary machine-learning (ML) methods~\cite{ml-time-ser-comp_aages10} 
	such as gradient boosted trees (GBT) and neural networks (NN), the essential complexity inhabiting in time series has been more and more recognized, the forecasting models having extended themselves from linear, Markovian assumptions to nonlinear, non-Markovian, and even more general situations~\cite{dyn-complex-syst_bmc98}. Among all known methods, recurrent NN architectures~\cite{lstm-empir_jzs15}, including plain recurrent neural networks (RNN)~\cite{rnn-mult_gsclc90} and long short-term memory (LSTM)~\cite{lstm_hs97}, are the most capable of capturing the complexity, as they admit the fundamental recurrent behavior of time series data. LSTM has proved useful on speech recognition and video analysis tasks~\cite{deep-learn_lbh15} where maintaining long-term memory is the essential complexity. To this objective, novel architectures such as higher-order RNN/LSTM (HO-RNN/LSTM)~\cite{ho-rnn_sj16} have been introduced to capture long-term non-Markovianity explicitly, further improving performance and leading to more accurate theoretical analysis. 
	
	Still, another dominance of complexity---chaos---has been far less understood. Even though enormous theory/data-driven studies on forecasting chaotic time series by recurrent NN have been conducted~\cite{rnn-chaos_kpv92,rnn-chaos_zx00,rnn-chaos_hxxy04,rnn-chaos_ll16,rnn-chaos_vbwsk18}, there is still no clear guidance of which features play the most general role in chaos forecasting methods. The notorious indication of chaos,
	\begin{eqnarray}
	\label{eq_lyapunov-exp}
	\left|\delta x_t\right|\approx e^{\lambda t}\left|\delta x_0\right|,
	\end{eqnarray}
	(where $\lambda$ denotes the spectrum of Lyapunov exponents) suggests that the difficulty of chaos forecasting is two-fold: first, any small error will propagate exponentially, and thus multi-step-ahead predictions will be exponentially worse than one-step-ahead ones; second, and more subtly, when the actual time difference $\Delta t$ between consecutive steps increases, the minimum redundancy needed for smoothly descending to the global minima during NN training also increases exponentially. Most studies on chaos forecasting only address the first difficulty by improving prediction accuracy at the global minima, yet the latter is indeed more crucial, especially when $\Delta t$ is so large that a trivial, ergodic local minimum would most likely be reached instead. 
	Recently, \emph{tensorization} has been introduced in recurrent NN architectures~\cite{tensor-rnn-video_ykt17,tensor-rnn-lang_ss18}. A tensorized version of HO-RNN/LSTM, namely HOT-RNN/LSTM~\cite{hot-rnn_yzay19}, has claimed its advantage in learning long-term nonlinearity on Lorenz systems of small $\Delta t$. On the one hand, we believe that the global minima of chaos (where dominance of linear dependence is absent) can be most efficiently reached by tensorization approaches, where all nonlinear terms, up to some polynomial order, are treated explicitly and weighted equally. On the other hand, for simple chaotic dynamical systems, nonlinear complexity is only encoded in the short term, not long term, which HO/HOT models will not be very useful in capturing when $\Delta t$ is large. Hence, a new tensorization-based recurrent NN architecture is desired so as to push our understanding of chaos in time series and meet practical needs, e.g., modeling laminar flame fronts and chemical oscillations~\cite{turbul-predict_r18,turbul-predict_phglo18,turbul-predict_jl19,turebul-predict_qm20}.
	
	In this paper, we introduce a new LSTM-based architecture by tensorizing the cell-state-to-state propagation therein, keeping the long-term memory feature of LSTM while simultaneously enhancing learning of short-term nonlinear complexity. Compared with traditional LSTM architectures including stacked LSTM~\cite{stack-lstm_g14} and other aforementioned statsitics/ML-based forecasting methods, our model is shown to be a general and the best approach for capturing chaos in almost every typical chaotic continuous-time dynamical system and discrete-time map with controlled comparable NN training conditions, justified by both our theoretical analysis and experimental results. 
	Our model has also been tested on real-world time series datasets where the improvements range up to $6.3\%$.

	During the tensorization, we have explicitly \emph{embedded} many-body quantum state structures---a way of reducing the exponentially large degree of freedom of a tensor (a.k.a. tensor decomposition)---from condensed matter physics, which is not uncommon for NN design~\cite{ml-phys_cccdstvmz19}. Since a many-body entangled state living in a tensor-product Hilbert space is hardly separable, this similarity motivates us to adopt a special measure of tensor complexity, namely, the entanglement entropy (EE)~\cite{area-law_ecp10}, which is commonly used in quantum physics and quantum information~\cite{ml-phys_cccdstvmz19}. For one-dimensional many-body states, two thoroughly studied, popular but different structures exist: multiscale entanglement renormalization ansatz (MERA)~\cite{mera_v08} or matrix product states (MPS)~\cite{mps_vc06}, of which the EE scales with the subsystem size or not at all, respectively~\cite{area-law_ecp10}. For most pertinent studies, 
	MPS has been proved efficient enough to be applicable to a variety of tasks~\cite{area-law-image_z17,express-power-rnn_kno18,mps-classif_bskj19,area-law-q-nn_jwwgg19}. However, our experiments show that, regarding our entanglement-embedded design of the new tensorized LSTM architecture, LSTM-MERA performs even better than LSTM-MPS in general without increasing the number of parameters. Our finding leads to another interesting result: we conjecture that not only should tensorization be introduced, but the tensor's EE has to scale with the system size as well---and hence MERA is more efficient than MPS at learning chaos.
	
	
	\begin{figure*}[t]
		\begin{minipage}[t]{59mm}
			\raggedleft
			\includegraphics[height=80mm]{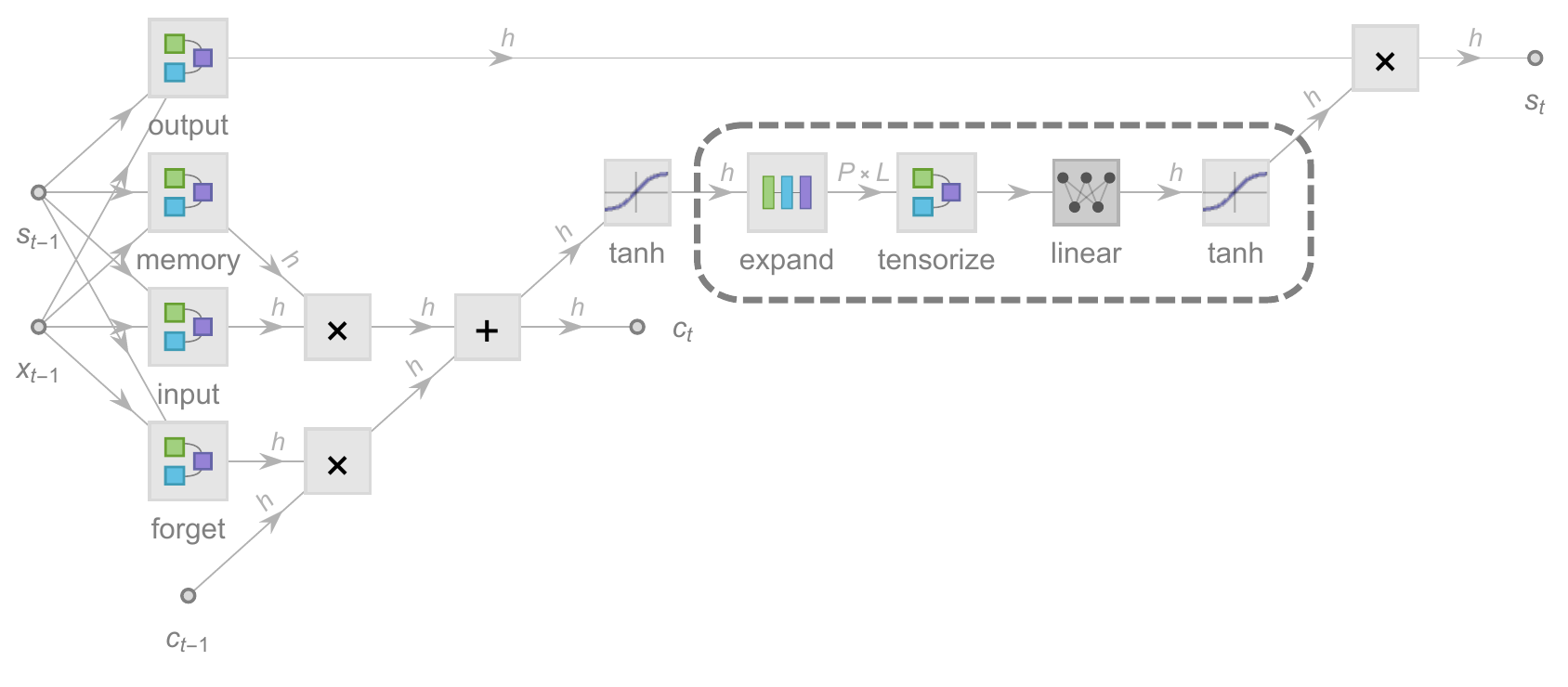}
		\end{minipage}
		\begin{minipage}[b]{119mm}
			\caption{\label{fig_netw}Architecture of a long short-term memory (LSTM) unit in the most common form of four gates: input, memory, forget, and output, enhanced by tensorized state propagation with four additional layers embedded (dashed rectangle): \emph{expand} [Eq.~(\ref{eq_tensor-operator})], \emph{tensorize} (Fig.~\ref{fig_tensorization}), \emph{linear}, and a \emph{tanh} activation function. $d$ is the input dimension of $x_t$, and $h$ is the hidden dimension of state $s_t$ and cell state $c_t$. An $h$-dimensional vector $\tanh c_t$ is first expanded into a $P\times L$-dimensional matrix where $L$ and $P$ are dubbed as the physical length and physical degrees of freedom (DOF), respectively. Then the matrix is tensorized into a $L$-rank tensor of dimension $P^L$ and passed forward. The effectiveness of this architecture is investigated in Section~\ref{sec_alternative}.}
		\end{minipage}
	\end{figure*}
	
	\section{Tensorized State Propagation}
	\subsection{Formalism}
	
	The formalism starts from an operator-theoretical perspective by defining two real \emph{operators} $\mathcal{W}$ and $\sigma$, from which most RNN architectures can be represented. $\mathcal{W}:\mathbb{X}\to\mathbb{G}$ is a linear operator, and $\sigma:\mathbb{G}\to\mathbb{G}$ is a nonlinear operator that $\sigma(G)=(\sigma\circ G)\in\mathbb{G}$ given $G\in\mathbb{G}$. Here $\oplus$ stands for tensor direct sum and $\circ$ for entry-wise operator product. All double-struck symbols ($\mathbb{X},\mathbb{G},\cdots$) used in the context are real vector spaces considered of \emph{covariant} type, as $\mathcal{W}$ can be interpreted as a linear-map-induced 2-\emph{contravariant} bilinear form. Next, a \emph{state propagation function (gate)} $g(x,y,\cdots;\mathcal{W})=\sigma(\mathcal{W}(x\oplus y\oplus\cdots))$ is introduced, where $X_i\in\mathbb{X}_i$. 
	Following the formalism, an LSTM is of the form 
	\begin{eqnarray}
	\label{eq_lstm-operator}
	s_t&=& g(1,x_{t-1},s_{t-1};\mathcal{W}_o)\circ \sigma(c_t),\qquad x_t=g(1,s_t;\mathcal{W}_x),\nonumber\\
	c_t&=& g(1,x_{t-1},s_{t-1};\mathcal{W}_f)\circ c_{t-1}\nonumber\\
	&+& g(1,x_{t-1},s_{t-1};\mathcal{W}_i)\circ g(1,x_{t-1},s_{t-1};\mathcal{W}_m),
	\end{eqnarray}
	where the four gates controlled by $\mathcal{W}_i$, $\mathcal{W}_m$, $\mathcal{W}_f$, and $\mathcal{W}_o$ are the input, memory, forget, and output gates in LSTM. Next, a realization of an LSTM is given by letting state $s_t$ and cell state $c_t$ be $h$-dimensional \emph{covectors} and input $x_t$ be a $d$-dimensional covector (Fig.~\ref{fig_netw}). Therefore $\mathcal{W}_i$ (also $\mathcal{W}_m$, $\mathcal{W}_f$, and $\mathcal{W}_o$) has a direct-sum contravariant realization $W_i\in\mathbf{M}(h,1)\oplus\mathbf{M}(h,d)\oplus\mathbf{M}(h,h)$ and contains $h(1+d+h)$ free real parameters at maximum. During NN training, only the free parameters of each linear operator $W$ are learnable, while all $\sigma$ (i.e., activation functions) are fixed to be {tanh}, {sigmoid} or other nonlinear functions. $c_t$ is known to suffer less from the vanishing gradient problem and thus captures long-term memory better, while $s_t$ tends to capture short-term dependence.
	
	Our tensorized LSTM architecture (Fig.~\ref{fig_netw}) is exactly based on Eq.~(\ref{eq_lstm-operator}), from which the only change is
	\begin{equation}
	\label{eq_tensor-lstm-operator}
	s_t= g(1,x_{t-1},s_{t-1};\mathcal{W}_o)\circ g(\mathcal{T}(\sigma(c_t));\mathcal{W}_{\mathcal{T}}).
	\end{equation}
	$g(\mathcal{T}(\sigma(c_t));\mathcal{W}_{\mathcal{T}})$ is called a \emph{tensorized state propagation function} as $\mathcal{W}_{\mathcal{T}}:\mathbb{T}\to\mathbb{G}$ acts on a covariant tensor,
	\begin{equation}
	\label{eq_tensor-operator}
	\mathcal{T}(\sigma(c_t))=\bigotimes_{l}\,\left(1\oplus q_{t,l}\right)=\bigotimes_{l}\,\left(1\oplus{\mathcal{W}_l}(\sigma(c_t))\right).
	\end{equation}
	Each $\mathcal{W}_l$ in Eq.~(\ref{eq_tensor-operator}) maps from the cell state $c_t$ to a new covector $q_{t,l}\in\mathbb{Q}$. Here, $\mathbb{Q}$ is named a \emph{local q-space}, as by way of analogy it is considered encoding the \emph{local} degree of freedom in quantum mechanics. $\mathbb{Q}$ can be extended to the complex number field if necessary. Mathematically, Eq.~(\ref{eq_tensor-operator}) offers the possibility of directly constructing orthogonal polynomials up to order $L$ from $\sigma(c_t)$ to build up nonlinear complexity. In fact, when $L$ goes to infinity,  $\mathbb{T}=\lim\limits_{L\to\infty}(1\oplus \mathbb{Q})^{\otimes L}=1\oplus \mathbb{Q} \oplus \mathbb{Q} \otimes \mathbb{Q}\oplus \cdots$ becomes a tensor algebra (up to a multiplicative coefficient), and $\mathcal{T}(\sigma(c_t))$ admits any nonlinear smooth function of $c_t$.
	
	Following the same procedure above, Eq.~(\ref{eq_tensor-operator}) is realized by choosing $L$ independent realizations, $W_l\in\mathbf{M}(P-1,h)$, $l=1,2,\cdots,L$, which in total contain $L(P-1)h$ learnable parameters at maximum,
	\begin{equation*}
	\label{eq_expand}
	\left(\begin{matrix}
	\left[\tanh c_t\right]_1\\\left[\tanh c_t\right]_2\\\vdots\\\left[\tanh c_t\right]_h
	\end{matrix}\right)
	\xrightarrow{\begin{matrix}
		\text{linear +}\\\text{padding}
		\end{matrix}}
	\left(\begin{matrix}
	\left(\begin{matrix}
	1\\\left[q_t\right]_{21}\\\left[q_t\right]_{31}\\\vdots\\\left[q_t\right]_{P1}
	\end{matrix}\right)
	&
	\left(\begin{matrix}
	1\\\left[q_t\right]_{22}\\\left[q_t\right]_{32}\\\vdots\\\left[q_t\right]_{P2}
	\end{matrix}\right)
	&
	\begin{matrix}
	\cdots\\\cdots\\\cdots\\\ddots\\\cdots
	\end{matrix}
	&
	\left(\begin{matrix}
	1\\\left[q_t\right]_{2L}\\\left[q_t\right]_{3L}\\\vdots\\\left[q_t\right]_{PL}
	\end{matrix}\right)
	\end{matrix}\right)
	,\quad
	\end{equation*}
	while letting $\mathcal{T}(\sigma(c_t))\equiv\mathcal{T}(\tanh c_t)$. From the realization of Eq.~(\ref{eq_tensor-lstm-operator}), $W_\mathcal{T}\in\mathbf{M}(h,P^L)$~[Fig.~\ref{fig_tensor}], however, a problem of exponential explosion (a.k.a. ``curse of dimensionality'') arises. Treating $W_\mathcal{T}$ maximally by training all $hP^L$ learnable parameters is very computationally expensive, especially as $L$ cannot be small because it governs the nonlinear complexity. To overcome this ``curse of dimensionality'', \emph{tensor decomposition} techniques have to be exploited~\cite{express-power-rnn_kno18} for the purpose of finding a much smaller subset $\mathbf{T}\subset\mathbf{M}(h,P^L)$ to which all possible $W_\mathcal{T}$ belong without sacrificing too much expressive power.

	\subsection{Many-Body Entanglement Structures}
	Below, we introduce in details the two many-body quantum state structures (MPS and MERA) as efficient low-order representation of $W_\mathcal{T}$. A comparison of these two entanglement structures is delivered after introducing an important measure of tensor complexity: the scaling behavior of EE.
	
	\subsubsection{MPS}
	As one of the most commonly used tensor decomposition techniques, MPS is also widely known as the tensor-train decomposition~\cite{tensor-train-spectr_bekm16} and takes the following form~[Fig.~\ref{fig_mps}]
	\begin{equation*}
	[W_\mathcal{T}]^h_{\mu_1\cdots\mu_L}=\sum_{\{\alpha\}}^{D_{\text{II}}}[w_0]_{\alpha_1\alpha_{L+1}}^{h}\left([w^\dagger_1]_{\mu_1}^{\alpha_1\alpha_2}[w^\dagger_2]_{\mu_2}^{\alpha_2\alpha_3}\cdots[w^\dagger_L]_{\mu_L}^{\alpha_L\alpha_{L+1}}\right)
	\end{equation*}
	in our model, where $w^\dagger_1,w^\dagger_2,\cdots,w^\dagger_L$ are learnable 3-tensors (where $^\dagger$ denotes symbolically that they are \emph{inverse isometries}~\cite{area-law_ecp10}). $D_{\text{II}}$ is an artificial dimension (the same for all $\alpha$). $w_0$ is no more than a linear transformation that collects the boundary terms and keeps symmetry. The above notations are used for consistency with quantum theory~\cite{area-law_ecp10} and the following MERA representation.
	
	\subsubsection{MERA}
	
	The best way to explain MERA is by graphical tools, e.g., tensor networks~\cite{area-law_ecp10}. MERA differs from MPS in its multi-level tree structure: each level $\{\text{I},\text{II},\cdots\}$ contains a layer of 4-tensor \emph{disentanglers} of dimensions $\{D_\text{I}^4,D_\text{II}^4,\cdots\}$ and then a layer of 3-tensor \emph{isometries} of dimensions $\{D_\text{I}^2\times D_\text{II},D_\text{II}^2\times D_\text{III},\cdots\}$, of which details can be found in Ref.~\cite{mera_v08}. MERA is similar to the Tucker decomposition~\cite{hierarch-tucker_g10} but fundamentally different because of the existence of disentanglers which smears the inhomogeneity of different tensor entries~\cite{mera_v08}.
	
	Figure~\ref{fig_mera} shows a reorganized version of MERA used in our model where the storage of independent tensors is maximally compressed before they being multiplied with each other by tensor products, which allows more GPU acceleration during NN training.
	
	\subsubsection{Scaling behavior of EE} 
	Given an arbitrary tensor $W_{\mu_1\cdots\mu_L}$ of dimension $P^L$ and a cut $l$ so that $1\le l\ll L$, the EE is defined in terms of the $\alpha$-R\'enyi entropy~\cite{area-law_ecp10},
	\begin{equation}
	\label{eq_ee}
	S_{\alpha}(l)\equiv S_{\alpha}(W(l))=\frac{1}{1-\alpha}\log\frac{\sum_{i=1}^{P^{l}}\sigma_i^{\alpha}(W(l))}{\left(\sum_{i=1}^{P^{l}}\sigma_i(W(l))\right)^{\alpha}},
	\end{equation}
	assuming $\alpha\ge1$. The Shannon entropy is recovered under $\alpha\to1$. 
	$\sigma_i(W(l))$ in Eq.~(\ref{eq_ee}) is the $i$-th singular value of matrix $W(l)=W_{(\mu_1\times\cdots\times\mu_{l}),(\mu_{l+1}\times\cdots\times\mu_L)}$, matricized from $W_{\mu_1\cdots\mu_L}$. How $S(l)$ scales with $l$ determines how much redundancy exists in $W_{\mu_1\cdots\mu_L}$, which in turn tells how efficient a tensor decomposition technique can be. For one-dimensional gapped low-energy quantum states, their EE saturates even as $l$ increases, i.e., $S_{\alpha}(l)=\Theta(1)$. Thus their low-entanglement characteristics can be efficiently represented by MPS, of which the EE does not scale with $l$ either and is bounded by $S_{\alpha}(l)\le S_1(l)\le2\log D_\text{II}$~\cite{area-law_ecp10}. By contrast, a non-trivial scaling behavior $S_{\alpha}(l)=\Theta(\log l)$ corresponds to gapless low-energy states and can only be efficiently represented by MERA, of which $S_{\alpha}(l)\le S_1(l)\le C+\sum_{\text{level}=1}^{\log_2 l}\log D_\text{level}\approx C+C' \log l$ scales logarithmically~\cite{mera_v08}. Both bounds of MPS and MERA are also proved to be tight.
	
	The different EE scaling behaviors between MERA and MPS have hence provided an apparent geometric advantage of MERA, i.e., its two-dimensional structure~[Fig.~\ref{fig_mera}], enlarging which will increase not only the width but also the \emph{depth} of NN as the number of applicable levels scale logarithmically with $L$, offering even more power of model generalization on the already-inherited depth of LSTM architecture~\cite{lstm-empir_jzs15}. Such an advantage is further confirmed by Eq.~(\ref{eq_ee_err}) and then in Section~\ref{sec_same-complex} where tested are tensorized LSTMs with the two different representations, LSTM-MPS and LSTM-MERA.	
	
	\begin{figure*}[t]
		\begin{minipage}[b]{58mm}
			\centering
			\includegraphics[width=58mm]{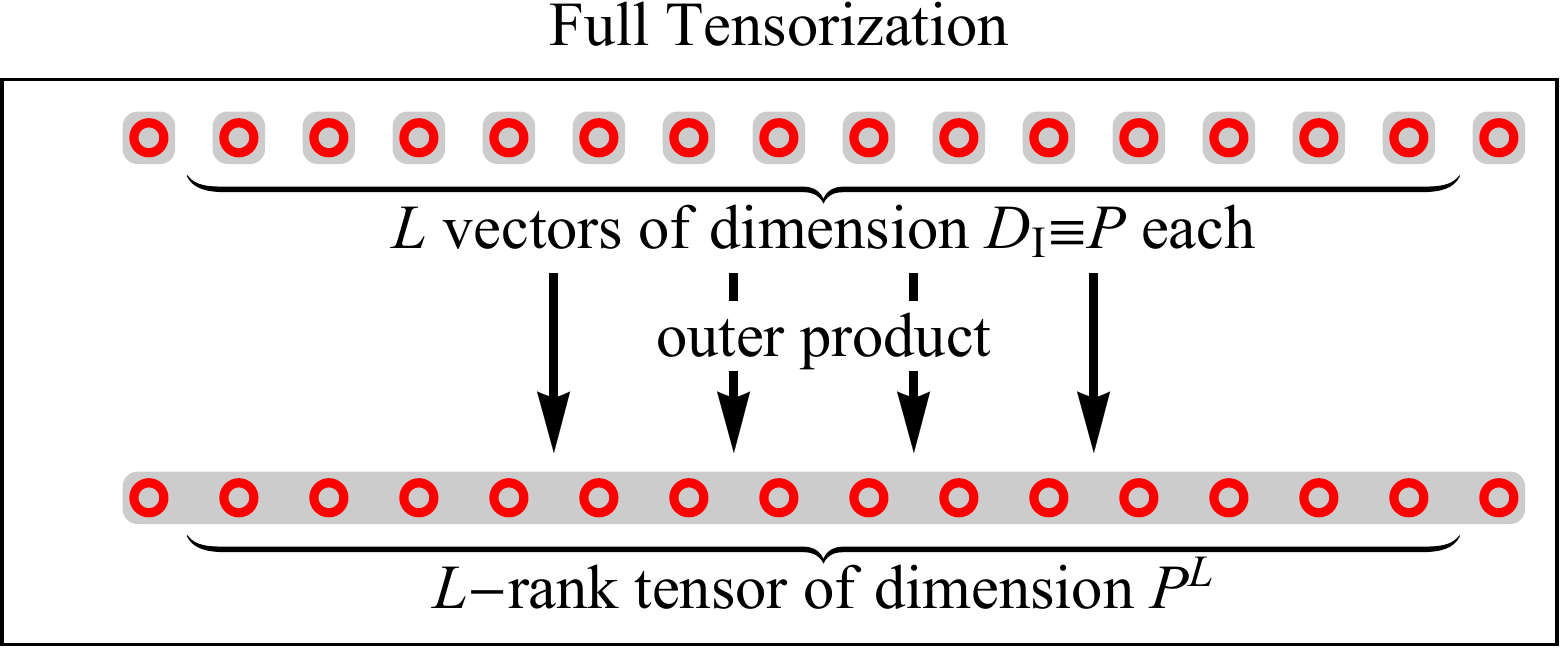}
			\subcaption{\label{fig_tensor}}
			\vspace{6.18mm}
			\includegraphics[width=58mm]{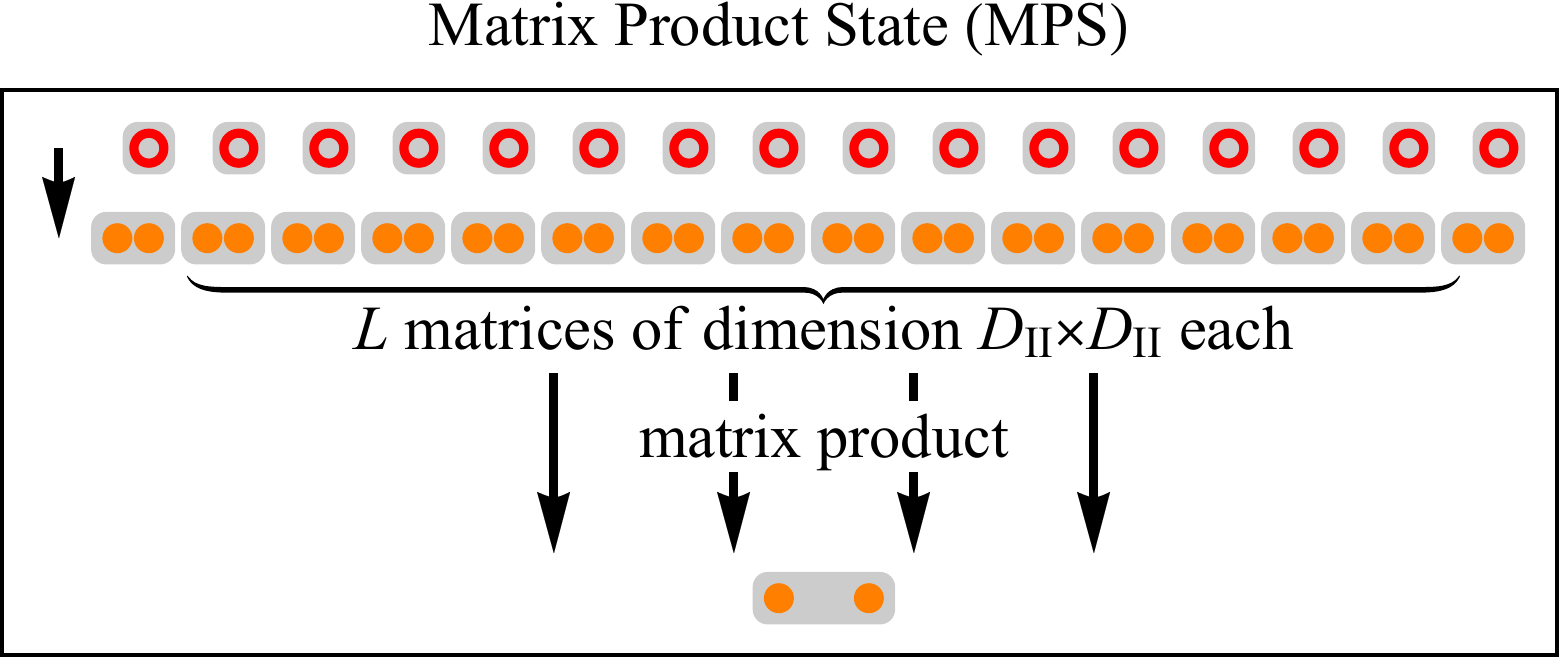}
			\subcaption{\label{fig_mps}}
		\end{minipage}
		\begin{minipage}[b]{58mm}
			\centering
			\includegraphics[width=58mm]{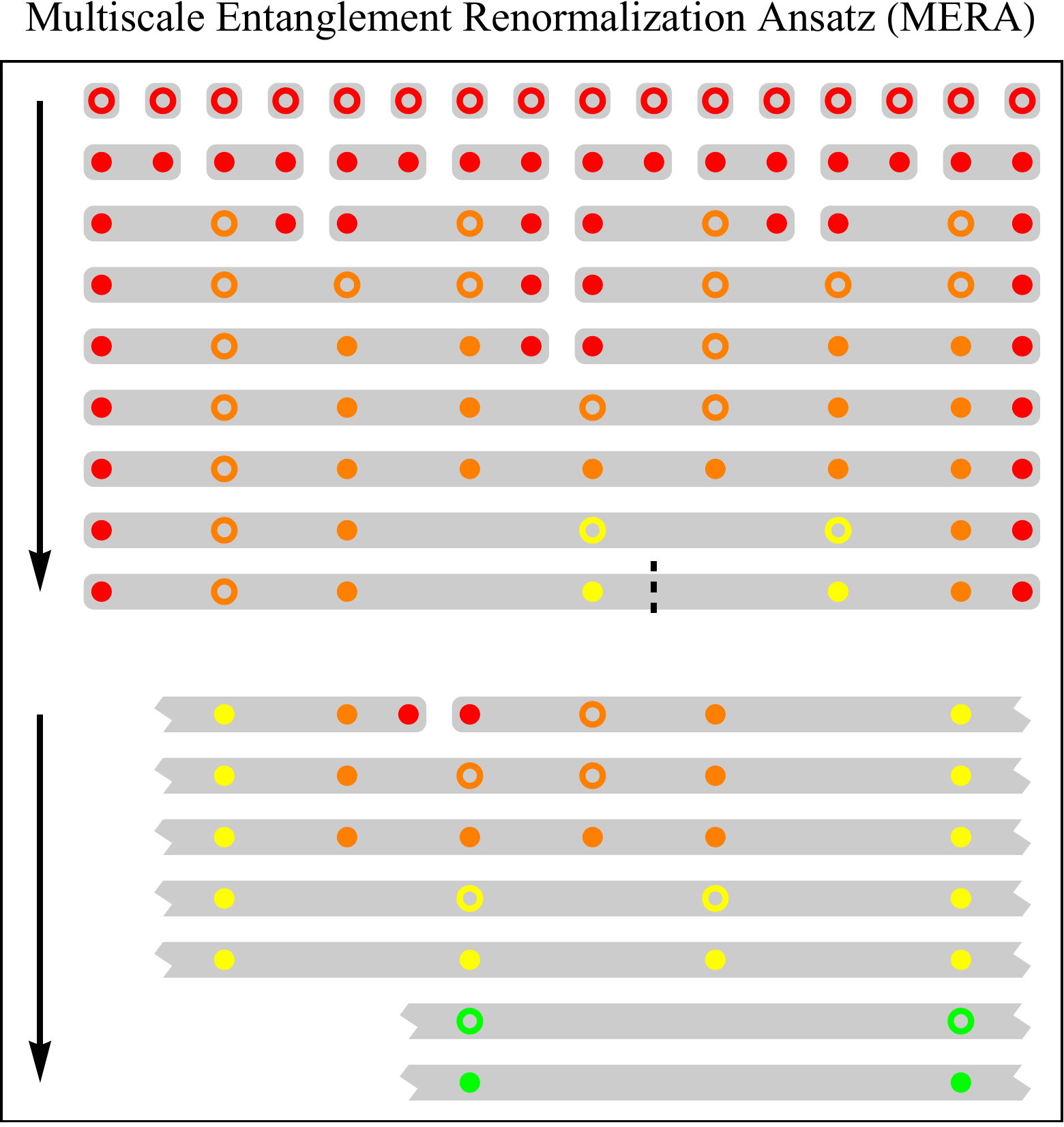}
			\subcaption{\label{fig_mera}}
		\end{minipage}
		\begin{minipage}[b]{58mm}
			\centering
			\includegraphics[width=58mm]{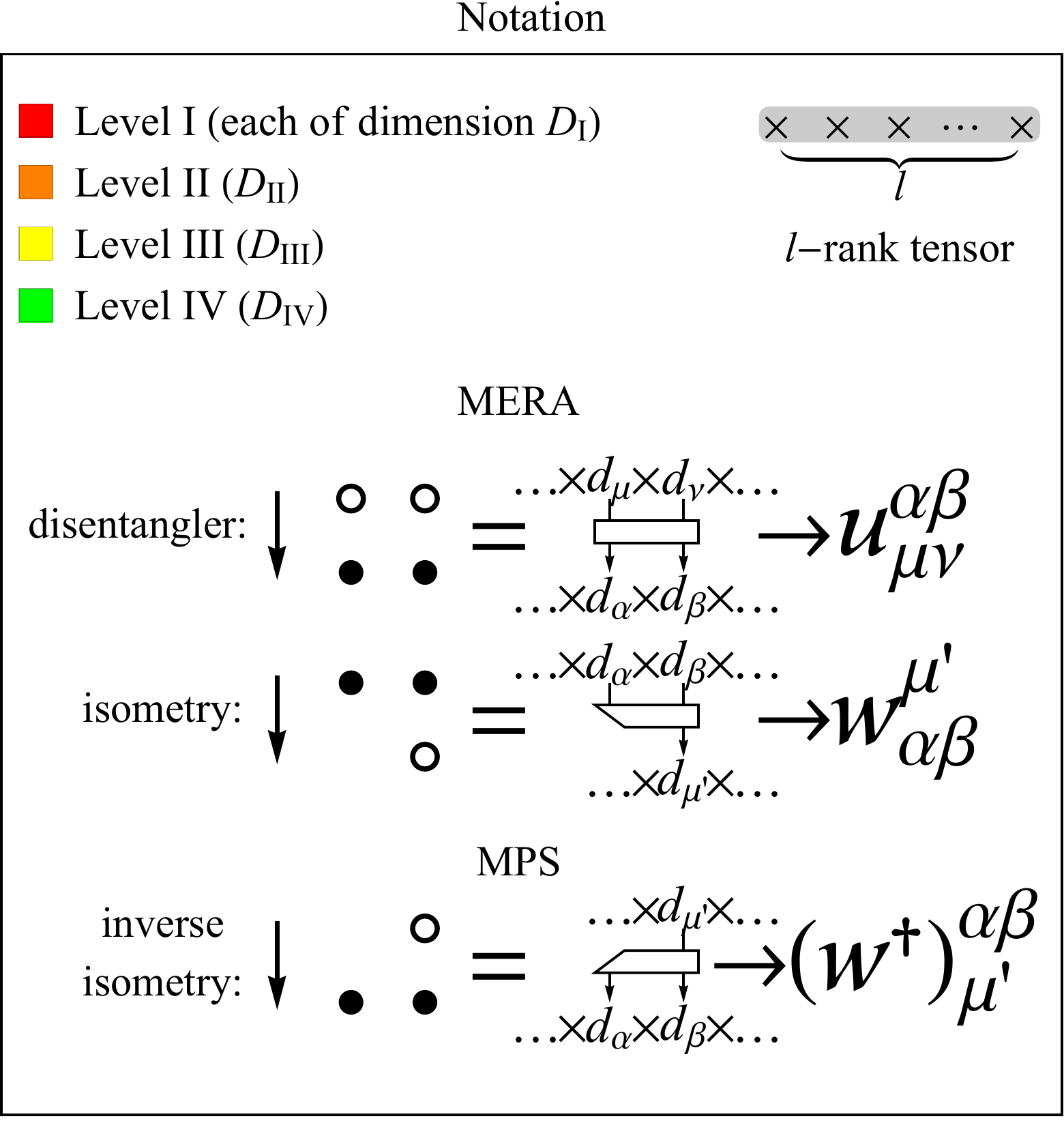}
			\subcaption{\label{fig_notation}}
		\end{minipage}
		
		\vspace{-2mm}
		\caption{\label{fig_tensorization}Tensorize layer: quantum entanglement structures. 
			\subref{fig_tensor},~Full tensorization. \subref{fig_mps},~Matrix product state (MPS). \subref{fig_mera},~Multiscale entanglement renormalization ansatz (MERA). The MPS and MERA are tensor representations that are widely used for characterizing many-body quantum entanglement in condensed matter physics. \subref{fig_notation},~Notations. A full tensor can be represented by introducing multiple auxiliary and learnable tensors (e.g., disentanglers and isometries as used in MERA and inverse isometries as used in MPS) of different virtual dimensions $\{D_\text{I},D_\text{II},\cdots\}$ labeled by different levels, rendered by different colors. The first-level virtual dimension is $D_\text{I}\equiv P$, the physical DOF by definition. Other virtual dimensions $\{D_\text{II},\cdots\}$ are free hyperparameters to be chosen, the larger which the better should the representation of the full tensor be. The numbers of applicable levels in \subref{fig_tensor} and \subref{fig_mps} are always constant (one and two, respectively), yet the number of applicable levels in \subref{fig_mera} is $\log_2 L$, relying on the physical length $L$.}
	\end{figure*}
	
	\section{Theoretical Analysis}
	\subsection{Expressive Power}
	Adding the tensorized state propagation function, Eq.~(\ref{eq_tensor-operator}), to the LSTM architecture is key for learning chaotic time series, as explained below.
	\begin{lemma}
		\label{lemma_chaos}
		Given an LSTM architecture of form Eq.~(\ref{eq_lstm-operator}) which produces a chaotic dynamical system $x_t$, characterized by a matrix $\lambda$ of which the spectrum is the Lyapunov exponent(s), i.e., any variation $\delta x_t$ propagates exponentially 
		[Eq.~(\ref{eq_lyapunov-exp})], then, up to the first order (i.e., $\delta x_{t-1}$),
		\begin{eqnarray}
		\label{eq_lemma_chaos}
		\left|\delta s_t\right|\ge C e^{\lambda}\left|\delta c_t\right|,
		\end{eqnarray}
		where $C\propto1/\|\mathcal{W}\|_{\infty}^{2}$ and $\|\cdot\|_{p=\infty}$ is the operator norm.
	\end{lemma}
	\begin{proof}
		From Eq.~(\ref{eq_lstm-operator}) one has $\delta x_t=(\partial g(1,s_t;\mathcal{W}_x)/\partial s_t)\delta s_t$ where the first-order derivative is bounded by $\|\partial g(1,s_t;\mathcal{W}_x)/\partial s_t\|_{\infty}\le\|\mathcal{W}_x\|_{\infty}\|\sigma'\|_{L_{\mu}^{\infty}}\le\|\mathcal{W}_x\|_{\infty}$, since the derivative of the active function supported on $(-\infty,\infty)$ satisfies $\|\sigma'\|_{L_{\mu}^{\infty}}\equiv\|1/\cosh^2\|_{L_{\mu}^{\infty}}\le 1$. On the other hand, one has
		\begin{eqnarray*}
			\delta c_t=\left[c_{t-1}\circ \partial g(1,x_{t-1},s_{t-1};\mathcal{W}_f)/\partial x_{t-1}+\partial\left(g(1,x_{t-1},s_{t-1};\mathcal{W}_i)\right.\right.\nonumber\\
			\left.\left. \circ\, g(1,x_{t-1},s_{t-1};\mathcal{W}_m)\right)/\partial x_{t-1}\right]\delta x_{t-1}+O(\delta x_{t-2})+\cdots,\qquad
		\end{eqnarray*}
		and thus
		\begin{equation*}
		|\delta c_t|\le \|\mathcal{W}_f\|_{\infty}\left|c_{t-1}\right|\circ \delta |x_{t-1}|+\left(\|\mathcal{W}_i\|_{\infty}+\|\mathcal{W}_m\|_{\infty}\right)\delta |x_{t-1}|
		\end{equation*}
		which produces Eq.~(\ref{eq_lemma_chaos}) supposing that all linear maps are of the same magnitude $\sim\|\mathcal{W}\|_{\infty}$. Note that $\left|c_{t-1}\right|$ is also bounded because $|g(1,x_{t-1},s_{t-1};\mathcal{W}_f)|\le1$ which means $c_t$ is stationary.
	\end{proof}
	Equation~(\ref{eq_lemma_chaos}) suggests that the state propagation from $c_t$ to $s_t$ carries the chaotic behavior. In fact, to preserve the long-term memorization in LSTM, $c_t$ has to depend on $c_{t-1}$ in a linear behavior and thus cannot carry chaos itself. This is further experimentally verified in Section \ref{sec_alternative}. Still, note that Eq.~(\ref{eq_lyapunov-exp}) is a necessary condition of chaos, not a sufficient condition.
	
	By virtue of the well-behaved polynomial structure of tensorization, the following theorem is proved for estimating the expressive power of the introduced tensorized state propagation function $\mathcal{T}(\sigma(c_t))$.
	\begin{theorem}
		Let  $f\in H_{\mu}^k(\Lambda)$ be a target function living in the $k$-Sobolev space, $H_{\mu}^k(\Lambda)=\left\{f\in L_{\mu}^2(\Lambda)\right|\left.\sum_{|\mathbf{i}|\le k}{\|\partial^{(\mathbf{i})}f\|_{L_{\mu}^2(\Lambda)}}<\infty\right\}$, where $\partial^{(\mathbf{i})}f\in L_{\mu}^2(\Lambda)$ is the $\mathbf{i}$-th weak derivative of $f$, up to order $k\ge0$, square-integrable on support $\Lambda=(-1,1)^{h}$ with measure $\mu$. $\mathcal{W}_{\mathcal{T}}\mathcal{T}(\sigma(c_t))$ can approximate $f(\sigma(c_t))$ with an $L_{\mu}^2(\Lambda)$ error at most
		\begin{equation}
		\label{eq_theorem_err}
		\|f-\mathcal{W}_{\mathcal{T}}\mathcal{T}\|_{L_{\mu}^2(\Lambda)}\le C \min(L,\lfloor L(P-1)/h\rfloor)^{-k}\|f\|_{H_{\mu}^k(\Lambda)}
		\end{equation}
		provided that $(h-1)hP^L\ge(h^{1+\min(L,\lfloor L(P-1)/h\rfloor)}-1)$. \space\space $\|f\|_{H_{\mu}^k(\Lambda)} =\sum_{|\mathbf{i}|\le k}{\|\partial^{(\mathbf{i})}f\|_{L_{\mu}^2(\Lambda)}}$ is the Sobolev norm and $C$ a finite constant.
	\end{theorem}
	\begin{proof}
		The \emph{H\"{o}lder-continuous spectral convergence theorem}~\cite{spectral-converg_cq82} states that $\|f-\text{P}_N f\|_{L_{\mu}^2(\Lambda)}\le C N^{-k}\|f\|_{H_{\mu}^k(\Lambda)}$, in which $\text{P}_N:L_{\mu}^2(\Lambda)\to \mathbb{P}_N$ is an orthogonal projection that maps $f$ to $\text{P}_Nf$. $\sigma(c(t))\in\Lambda$ is guaranteed as $\sigma\equiv\tanh$. The Sobolev space $\mathbb{P}_N\subset L_{\mu}^2(\Lambda)$ is spanned by polynomials of degree at most $N$. Next, note that in the realization of $\mathcal{T}(\sigma(c_t))$ each $W_l$ is independent [Eq.~(\ref{eq_tensor-operator})], and thus $\mathbb{P}_N=\text{span}(\mathcal{T}(\sigma(c_t)))$ is possible, where $N$ is determined by $L$, $P$, and $h$. When $P-1\ge h$, the maximum polynomial order is guaranteed $L$; when $P-1<h$, $\dim\{\mathbb{Q}\}<\dim\{\mathbb{G}\}$, and hence $\mathcal{T}(\sigma(c_t))$ can only fully cover polynomial order of up to $\lfloor L(P-1)/h\rfloor$. Finally, Eq.~(\ref{eq_theorem_err}) is proved from the fact that $\mathcal{W}_{\mathcal{T}}\mathcal{T}$ maximally admits $\text{P}_N f$ 
		as long as $hP^L\ge\sum_{i=0}^N h^i=(h^{N+1}-1)/(h-1)$, the latter of which is the size of the maximum orthogonal polynomial basis admitted by $\mathbb{P}_N$.
	\end{proof}
	Equation~(\ref{eq_theorem_err}) can be used to estimate how $L$ scales with the chaos the dynamical system possesses. In particular, Eq.~(\ref{eq_lemma_chaos}) suggests that $\partial^{(1)}f\sim e^{\lambda \Delta t}$ where $\Delta t$ is the actual time difference between consecutive steps. Therefore, to persevere the error bound [Eq.~(\ref{eq_theorem_err})] one at least expects $L^{-1}\sim e^{-\lambda  \Delta t}$, i.e., $L$ has to increase exponentially with respect to $\lambda  \Delta t$, to achieve which, tensorization is undoubtedly the most efficient way especially when $\Delta t$ is large.
	
	\begin{figure}[t]
		\raggedright
		\begin{minipage}[t]{86mm}
			\centering
			\includegraphics[width=50mm]{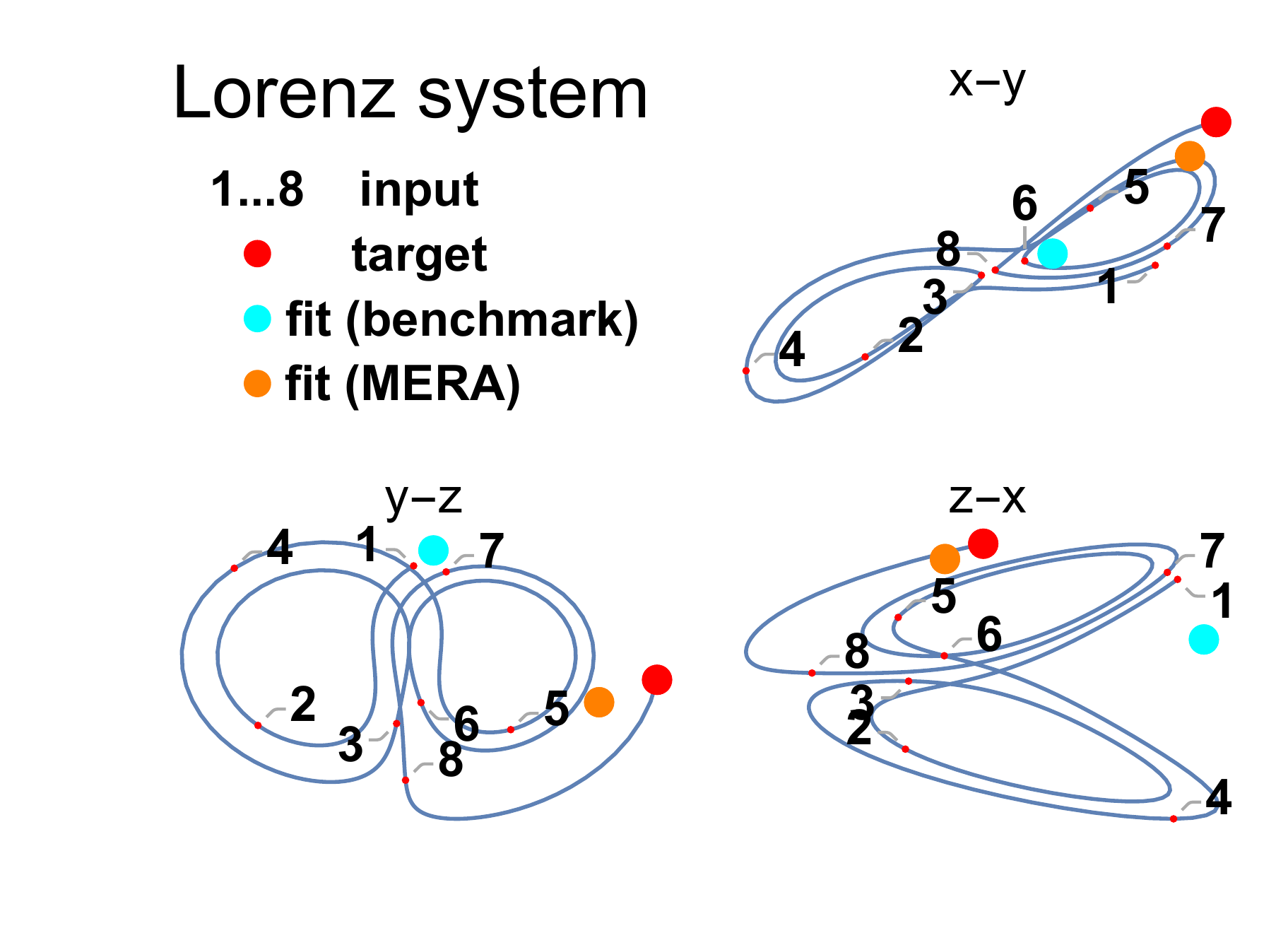}
			\begin{tabular}{lcccccc}
				\hline
				LSTM  & \# of param. & $h$ & $L$ & $P$ & $\{D_{\text{I}},D_{\text{II}},\cdots\}$ & RMSE \\
				\hline
				Benchmark  & $332$ & $7$ & -- & -- & -- & $0.307$ \\
				``Wider''  & $696$ & $11$ & -- & -- & -- & $0.279$ \\
				``Deeper''  & $640$ & $7$ & -- & -- & -- & $0.105$ \\
				MPS & $663$ & $7$ & $2^3$ & $2$ & $\{P,4\}$ & $0.088$ \\
				\textbf{MERA} & \boldmath$640$ & \boldmath$7$ & \boldmath$2^3$ & \boldmath$2$ & \boldmath$\{P,2,3\}$ & \boldmath$0.066$ \\
				\hline
			\end{tabular}
			\subcaption{\label{fig_lorenz}}
		\end{minipage}
		\hspace{4mm}
		\begin{minipage}[t]{86mm}
			\centering
			\includegraphics[width=50mm]{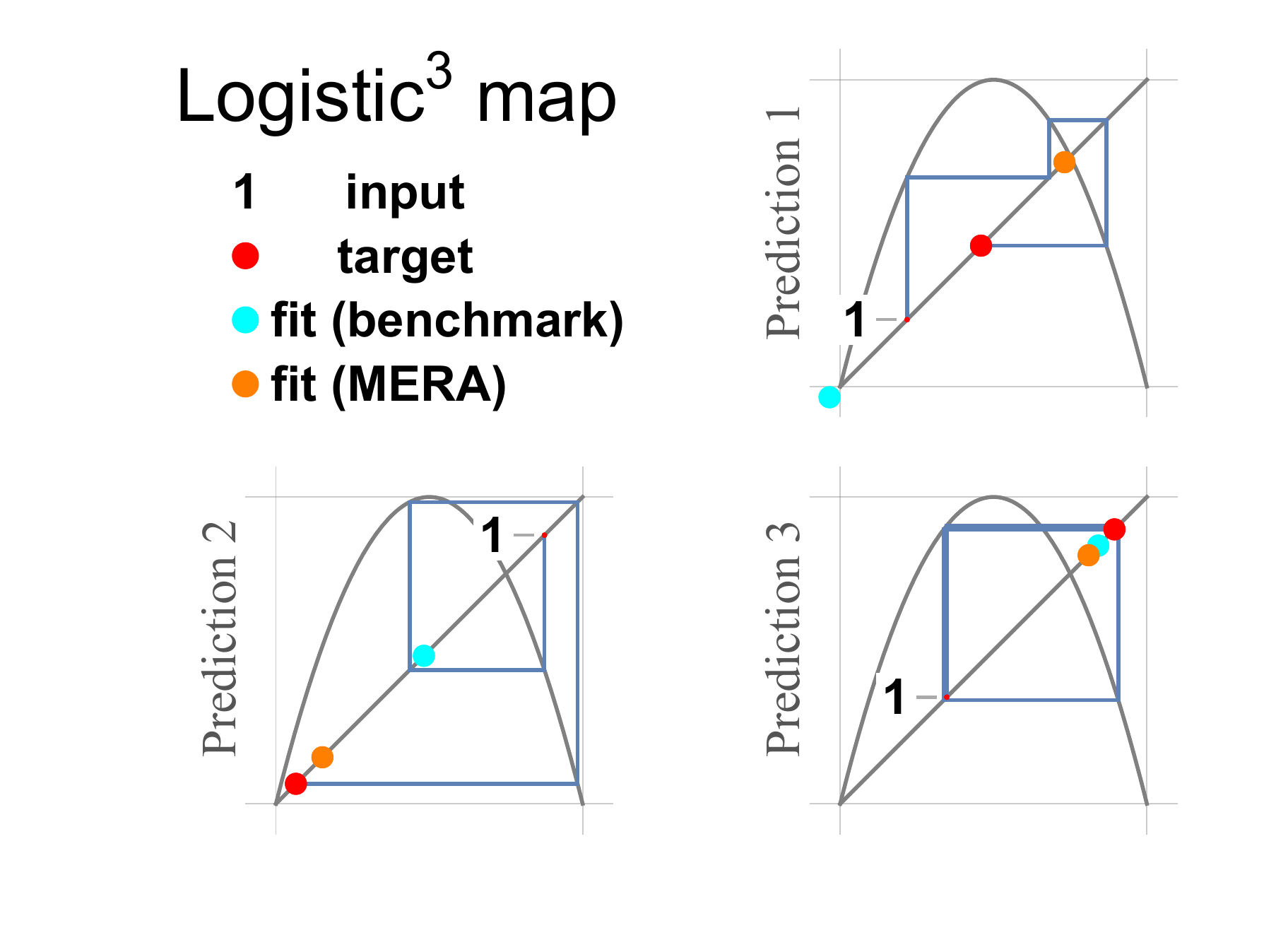}
			\begin{tabular}{lcccccc}
				\hline
				LSTM  &  \# of param. & $h$ & $L$ & $P$ & $\{D_{\text{I}},D_{\text{II}},\cdots\}$ & RMSE \\
				\hline
				Benchmark &  $35$ & $2$ & -- & -- & -- & $0.259$ \\
				``Wider'' &  $1169$ & $16$ & -- & -- & -- & $0.187$ \\
				``Deeper'' &  $1156$ & $11$ & -- & -- & -- & $0.204$ \\
				MPS&  $1231$ & $2$ & $2^3$ & $2$ & $\{P,9\}$ & $0.181$ \\
				\textbf{MERA} & \boldmath$1053$ & \boldmath$2$ & \boldmath$2^3$ & \boldmath$2$ & \boldmath$\{P,4,4\}$ & \boldmath$0.010$ \\
				\hline
			\end{tabular}
			\subcaption{\label{fig_logistic}}
		\end{minipage}
		
		\vspace{-2.0mm}
		\caption{\label{fig_learning-curves}Comparison of different LSTM-based architectures. \subref{fig_lorenz},~Lorenz system is a three-dimensional continuous-time dynamical system notable for its chaotic behavior. Discretization: $\Delta t=0.5$. Input steps~$=8$, training : validation : text $=2400:600:2000$, and number of epochs~$=120$ for all models. \subref{fig_logistic},~Logistic ``cubed'' map, i.e., a logistic map re-sampled every three steps. Input steps~$=1$, training : validation : text $=8000:2000:500$, and number of epochs~$=200$ for all models. Note that unlike continuous-time dynamical systems, chaos in discrete maps is more intrinsic and thus should be generally harder to learn. For example, the continuous counterpart of logistic map, a.k.a. the logistic differential equation does not exhibit any chaotic behavior.}
		\vspace{-1.0mm}
	\end{figure}
	
	\subsection{Worst-Case Bound by EE}
	The above analysis emphasizes the expressive power of tensorization. Now we compare the two different entanglement structures. A major difference between MPS and MERA is their EE scaling behaviors. We therefore proceed via the following theorem, relating the tensor approximation error and entanglement scaling.
	
	\begin{theorem}
		Given a tensor $[W_\mathcal{T}]_{\mu_1\cdots\mu_L}$ and its tensor decomposition $\overline{W}_\mathcal{T}$, the worst-case $p$-norm ($p\ge1$) approximation error is bounded from below by
		\begin{align}
		\label{eq_ee_err}
		&\min_{\{\overline{W}_\mathcal{T}\}}\max_{l \ge 1}\|W_\mathcal{T}(l)-\overline{W}_\mathcal{T}(l)\|_{p}\nonumber\\
		\ge&\min_{\{\overline{W}_\mathcal{T}\}} \max_{l \ge 1} \left|
		e^{\frac{1-p}{p}S_{p}(W_\mathcal{T}(l))} \|W_\mathcal{T}(l)\|_1 -
		e^{\frac{1-p}{p}S_{p}(\overline{W}_\mathcal{T}(l))} \|\overline{W}_\mathcal{T}(l)\|_1\right|,
		\end{align}
		where $S_{\alpha\equiv p}(W(l))$ is the $\alpha$-R\'enyi entropy [Eq.~(\ref{eq_ee})].
	\end{theorem}
	\begin{proof}
		Equation~(\ref{eq_ee_err}) is easily proved by noting the Minkowski inequality $\|A+B\|_p\le\|A\|_p+\|B\|_p$ and that $(1-\alpha)S_{\alpha}(l)=\alpha \log  \|W_\mathcal{T}(l)\|_{\alpha}-\alpha  \log  \|W_\mathcal{T}(l)\|_{1}$ when $\alpha\equiv p\ge1 $~[Eq.~(\ref{eq_ee})].
	\end{proof}
	
	The worst-case bound [Eq.~(\ref{eq_ee_err})] is optimized whenever $S_{p}(\overline{W}_\mathcal{T}(l))$ scales the same way as $S_{p}(W_\mathcal{T}(l))$ does. Assuming $S_{p}(W_\mathcal{T}(l))=C+C'\log l$, then an MPS-type $\overline{W}_\mathcal{T}$ cannot efficiently approximate ${W}_\mathcal{T}$ unless $D_\text{II}$ increases with $\log l$ too, from which the total number of free parameters $\sim PLD_\text{II}^2$ [Fig.~\ref{fig_mps}] however becomes unbounded. By contrast, a MERA-type $\overline{W}_\mathcal{T}$ matches the scaling, by which the total number of free parameters $\sim(D^4+D^3)L$ (where $D\equiv D_\text{II},\cdots= \exp C'$) is efficient enough toward any worst case $l$.

	\begin{figure*}[t!]
		\centering
		\hspace*{-10mm}  \includegraphics[width=170mm]{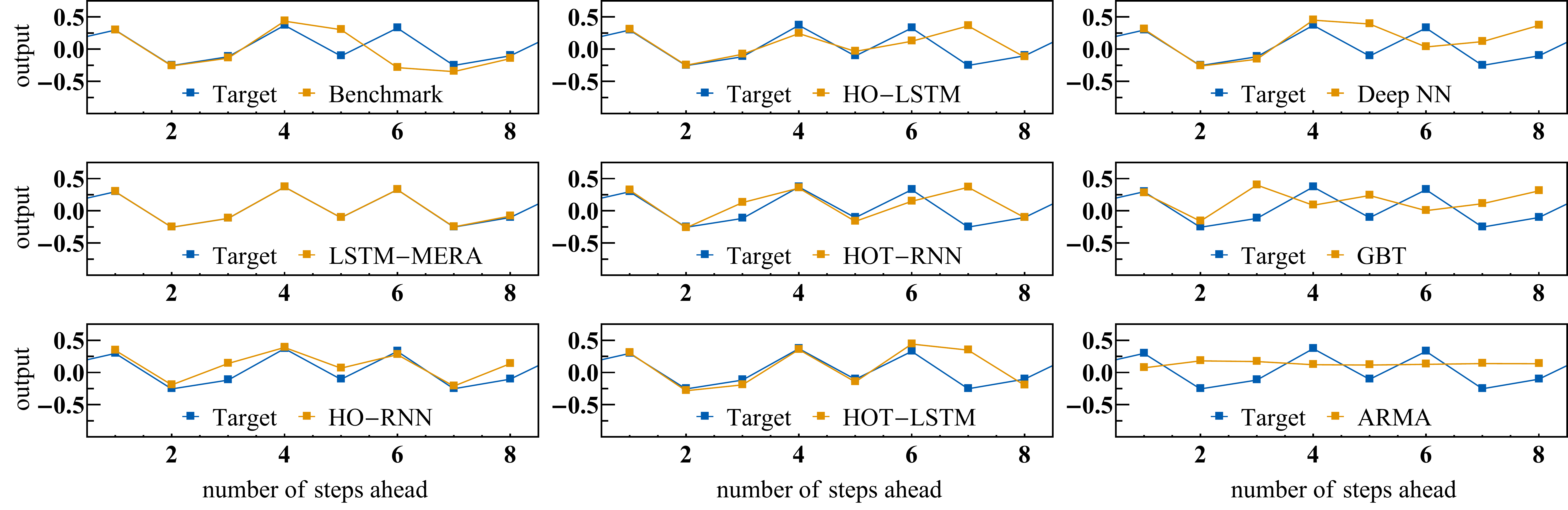}
		\centering
		\begin{minipage}[t]{86mm}
			{	
				\vspace{-2mm}
				\begin{tabular}{l|c|rrr}		
					\toprule
					Model &\# of param.& \multicolumn{3}{c}{RMSE ($\times 10^{-2}$)}\\
					($n$ steps ahead) & & 1 step& 2 steps& 4 steps\\
					\midrule
					Benchmark &35&$1.54$&$7.63$& $32.03$\\
					\textbf{LSTM-MERA} &\textbf{89}&\boldmath$0.19$&\boldmath$0.89$&\boldmath$13.77$\\
					HO-RNN &23&$11.91$&$23.16$&$27.69$\\
					HO-LSTM &83&$2.96$&$14.50$&$47.99$\\
					HOT-RNN &81&$12.04$&$23.76$&$29.61$\\
					HOT-LSTM &315&$1.39$&$6.40$&$26.83$\\
					\midrule
					Deep NN &17950&$0.81$&$3.66$&$23.49$\\
					\midrule
					GBT&--&$3.15$&$16.25$&$31.37$\\
					\midrule
					ARMA&--&$24.95$&$24.93$&$23.54$\\
					\bottomrule
				\end{tabular}
			}
		\end{minipage}
		\begin{minipage}[t]{86mm}
			\vspace{-2mm}
			\caption{\label{fig_gauss}Comparison of different statistical/ML models on Gauss ``cubed'' map, i.e., a Gauss iterated map  re-sampled every three steps. 
				Note that the Gauss iterated map is a one-dimensional chaotic map of which the dynamics is smoother than the logistic map and should be easier to learn. Input steps~$=8$. For RNN-based models, $h=2$, $L=2^2$, $P=2$, $\{D_\text{I},D_\text{II},\cdots\}=\{P,2\}$, training : validation : text $=8000:2000:500$, and number of epochs~$=200$. The explicit ``history'' length used in HO-RNN/LSTM~\cite{ho-rnn_sj16} and HOT-RNN/LSTM~\cite{hot-rnn_yzay19} is also $L$, and the tensor-train ranks are all $D_\text{II}$. Deep NN: depth~$=8$~($=$~input steps). GBT: maximum depth~$=8$. ARMA family: ARMA$(3,4)$.}
		\end{minipage}
	\end{figure*}
	
	\begin{figure}[t]
		\centering
		{
			\includegraphics[width=85mm]{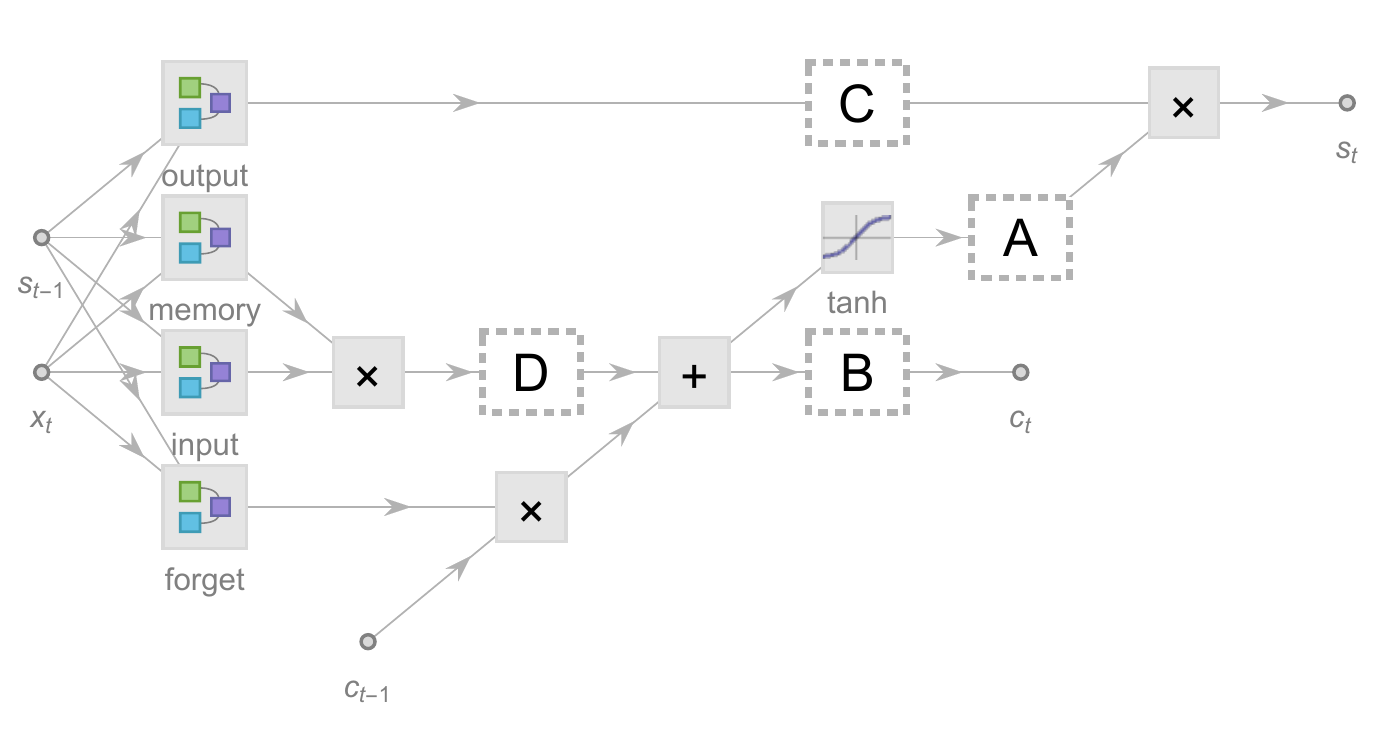}
		}
		\centering
		{
			\begin{tabular}{l|cc}		
				\toprule
				\multicolumn{3}{c}{{Thomas' cyclically symmetric dynamical system}}\\
				\midrule
				Model &Site&RMSE ($\times 10^{-1}$)\\
				\midrule
				Benchmark&&$1.13$\\
				\midrule
				\textbf{LSTM-MERA}&\textbf{A}&\boldmath${0.45}$\\
				\midrule
				Alternatives & B&$1.12$\\
				&C&$1.10$\\
				&D&$0.73$\\
				\bottomrule
			\end{tabular}
		}
		\caption{\label{fig_thomas}Comparison of LSTM-MERA (where the additional layers from Fig.~\ref{fig_netw} are located at Site~A) with its alternatives (where the additional layers are instead located at Site~B,~C~or~D), tested on Thomas' cyclically symmetric system, a three-dimensional chaotic dynamical system known for its cyclic symmetry $\mathbf{Z}/3\mathbf{Z}$ under change of axes. Discretization: $\Delta t=1.0$. Input steps~$=8$, $h=4$, $L=2^4$, $P=4$, $\{D_\text{I},D_\text{II},\cdots\}=\{P,2,2,4\}$, training : validation : text $=2400:600:2000$, and number of epochs~$=40$.}
	\end{figure}
	
	\section{Results}
	\label{sec_result}
	We investigate the accuracy of LSTM-MERA and its ability of generalization on 
	different chaotic time series datasets 
	by evaluating the root mean squared error (RMSE) of its one-step-ahead predictions against target values. The benchmark for comparison was chosen to be a vanilla LSTM of which the hidden dimension $h$ was arbitrarily chosen in advance.
	LSTM-MERA (and other architectures if present) was built upon the benchmark.
	
	
	Each time series dataset for training/testing consisted of a set of $N_X$ time series, $\{X^i|i=1,2,\cdots,N_X\}$. Each time series $X^i=\{x^i_t|t\in T^i\}$ is of fixed length  $|T^i|=\text{input steps }+1$ so that all but the last step of $X^i$ are input, while the last step is the one-step-ahead target to be predicted. The dataset $\{X^i\}$ was divided into two subsets, one for testing, and one for training which was further randomly split into a plain training set and a validation set by $80\%:20\%$. Complete details are given in Appendix~\ref{sec_data-process}.
	
	All models were trained by Mathematica 12.0 on its NN infrastructure, Apache MXNet, using an ADAM optimizer with $\beta_1=0.9$, $\beta_2=0.999$, and $\epsilon=10^{-5}$. Learning rate~$=10^{-2}$ and batch size~$=64$ were \emph{a priori} chosen. The NN parameters producing the lowest validation loss during the entire training process were accepted.

	\subsection{Comparison of LSTM-Based Architectures}
	\label{sec_same-complex}
	
	When evaluating the advantage of LSTM-MERA, a controlled comparison is essential to confirm that the architecture of LSTM-MERA is \emph{inherently} better than other architectures, not just because the increase of the number of free and learnable parameters (even though more parameters do \emph{not} necessarily mean more learning power). Here, we studied different architectures  (Fig.~\ref{fig_learning-curves}) that were all built upon the LSTM benchmark and shared nearly the same number of parameters (\# of param.). A ``wider'' LSTM was simply built by increasing $h$. A ``deeper'' LSTM was built by stacking two LSTM units as one unit.  
	In particular, LSTM-MPS and LSTM-MERA were built and compared.
	
	
	\subsubsection{Lorenz system}
	Figure~\ref{fig_lorenz} describes the forecasting task on the \emph{Lorenz system} and shows training results of the LSTM-based models. $\Delta t=0.5$ was chosen for discretization, which was large enough that the resultant time series hardly exhibited any pattern without the help of a phase line [Fig.~\ref{fig_lorenz}, \mbox{input~$1$-$8$}].
	
	In general, non-tensorized LSTM models performed worse than tensorized LSTM models. After the number of free parameters increased from $332$ (benchmark) to $668\pm28$, both the ``wider'' and ``deeper'' LSTMs showed signs of overfitting as the loss and validation learning curves deviated. The ``deeper'' LSTM yielded lower RMSE than the ``wider'' LSTM, confirming the common sense that a deep NN is more suitable of generalization than a wide NN.
	
	Both LSTM-MPS and LSTM-MERA yielded better RMSE and showed no sign of overfitting. However, LSTM-MERA was more powerful, showing an improvement of $\sim25\%$ than LSTM-MPS in RMSE [Fig.~\ref{fig_lorenz}]. The learning curve of LSTM-MERA was also steeper given the same number of epochs during learning. Note that the learning curves of LSTM-MERA and LSTM-MPS were, in general, not as smooth as non-tensorized LSTM models, implying the difficulty of learning nonlinear complexity during which the back-propagated gradient is usually highly irregular.
	
	\subsubsection{Logistic map}
	Figure~\ref{fig_logistic} describes a specific forecasting task on the simplest one-dimensional discrete-time map---the \emph{logistic map}: predicting the target given only a three-step-behind input. Different LSTM models yielded very different results when learning this complex task. After the number of free parameters increased from $35$ (benchmark) to $1142\pm89$, all LSTM models yielded lower RMSE than the benchmark. Interestingly, the ``deeper'' LSTM learned more slowly even than the benchmark and did not reach stable RMSE within $200$ epochs. LSTM-MPS was able to learn quickly than other non-tensorized LSTM models at the beginning but then reached plateaus and struggled to further decrease its RMSE. Only LSTM-MERA was able to reach a much lower RMSE with a remarkable improvement of $\sim94\%$ than LSTM-MPS [Fig.~\ref{fig_logistic}]. Instead of being stable, the learning curve of LSTM-MERA became very spiky after descending below certain values of RMSE which were previously learning barriers to the other LSTM models. We infer that the plateaus of RMSE reached by the other LSTM models might correspond to the infinite numbers of unstable quasi-periodic cycles in the chaotic phases. In fact, as shown in Fig.~\ref{fig_logistic}, Prediction~3, the benchmark fit the target better than LSTM-MERA for this specific example of a quasi-period-2 cycle. However, only did LSTM-MERA learn the full chaotic behavior and thus performed much better on general examples.
	
	The learning process for the logistic map task was indeed very random, and different realizations had yielded very different results. In many realizations, non-tensorized LSTM models did not even learn any patterns at all. By contrast, tensorized LSTM models were more stable in learning.
	
	\subsection{Comparison with Statistical/ML Models}
	\label{sec_statistical}
	We compared LSTM-MERA with more general models including traditional statisical and ML models including RNN-based architectures (Fig.~\ref{fig_gauss}). Especially, we looked into HOT-RNN/LSTM which also claimed to be able to learn chaotic dynamics (e.g. Lorenz system) through tensorization~\cite{hot-rnn_yzay19}. Furthermore, for each model we fed its one-step-ahead predictions back so as to make predictions for the second step, and kept feeding back and so on. In theory, the prediction error at the $t$-th step should increase exponentially with $t$ for chaotic dynamics  [Eq.~(\ref{eq_lyapunov-exp})].
	
	\subsubsection{Gauss iterated map}
	We tested the one-step-ahead learning task on the \emph{Gauss ``cubed'' map} on plain HO-RNN/LSTM~\cite{ho-rnn_sj16} and its tensorized version HOT-RNN/LSTM~\cite{hot-rnn_yzay19}. The explicit ``history'' length was chosen to be equal to our physical Length $L$. The tensor-train ranks were all chosen to be equal to $D_\text{II}$, the same as how we built the MPS structure in LSTM-MPS.
	
	Figure~\ref{fig_gauss} shows that neither HO-RNN nor HO-LSTM performed better than the benchmark, suggesting that introducing explicit non-Markovian dependence (Appendix~\ref{sec_ho-hot}) is \emph{not} helpful for capturing chaotic dynamics where the existing nonlinear complexity is never long-term. HOT-LSTM was better than the benchmark because of its MPS structure, suggesting that tensorization, on the other hand, is \emph{indeed} helpful for forecasting chaos. LSTM-MERA was still the best, with an improvement of $\sim88\%$ over the benchmark. Interestingly, the benchmark itself as a vanilla LSTM was already much better than plain RNN architectures (HO-/HOT-RNN).
	
	The learning task was next tested on fully connected deep NN architectures of depth~$\le8$ (equal to the input steps). At each depth three units were connected in series: a linear layer, a scaled exponential linear unit, and a dropout layer. Hyperparameters were determined by optimal search. The best model having the lowest validation loss consisted of $17950$ free parameters. The task was also tested on GBT of maximum depth~$=8$, as well as on ARMA family (ARMA, ARIMA, FARIMA, and SARIMA) among which the best statistical model selected out by Kalman filtering was $\text{ARMA}(3, 4)$.
	
	With enough parameters, the deep NN became the second best (Fig.~\ref{fig_gauss}). All RMSE increased when making longer-step-ahead predictions, and for the four-step-ahead task the deep NN and LSTM-MERA were the only models that did not overfit and still performed better than the statistical model, ARMA, which made no learning progress but only trivial predictions.
	
	\begin{figure*}[t]
		\centering
		\begin{minipage}[t]{51mm}
			\includegraphics[width=50mm]{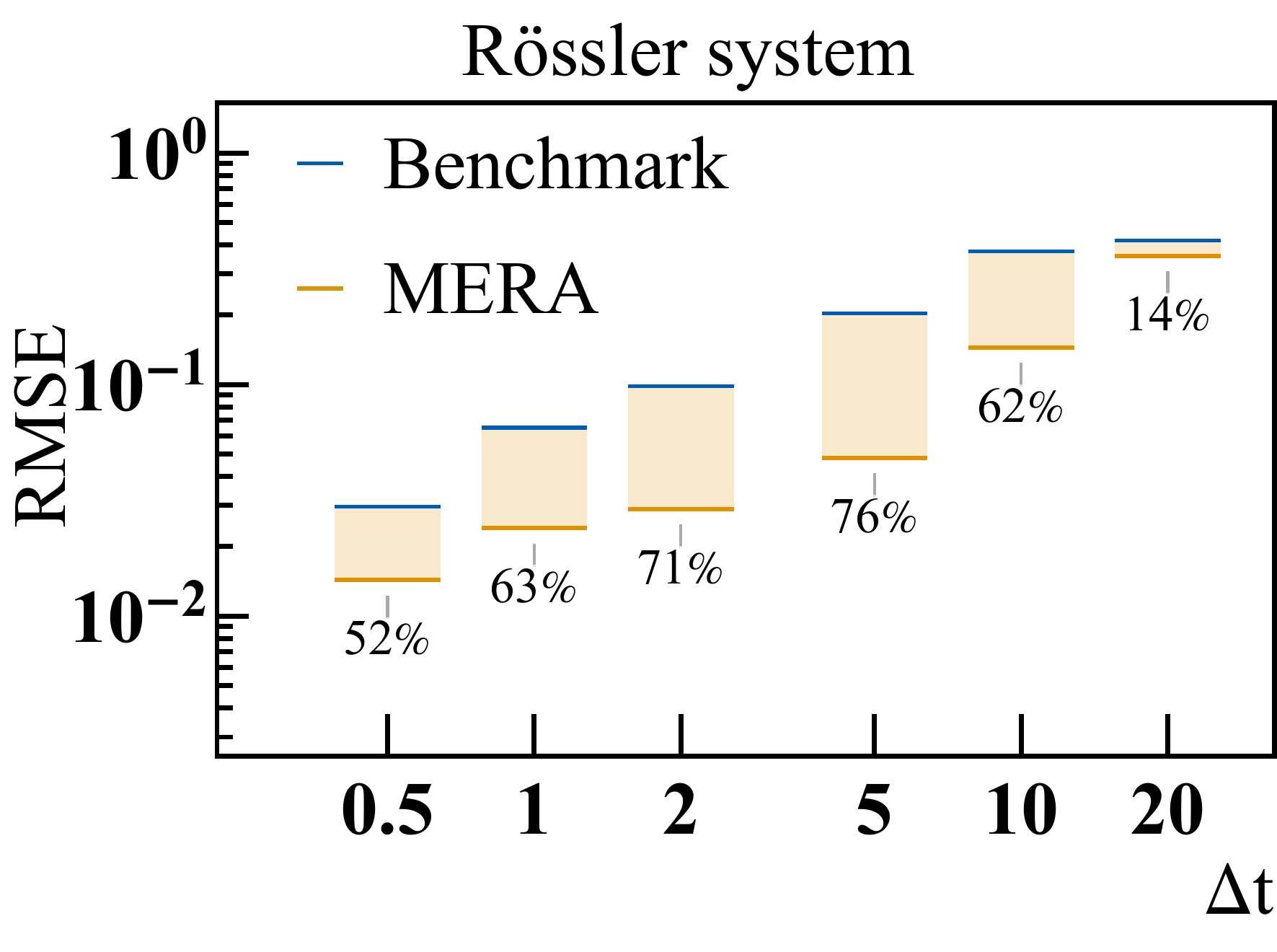}
			\vspace{-2mm}
			\subcaption{\label{fig_rossler}}
		\end{minipage}
		\addtocounter{subfigure}{1}
		\begin{minipage}[t]{51mm}
			\includegraphics[width=50mm]{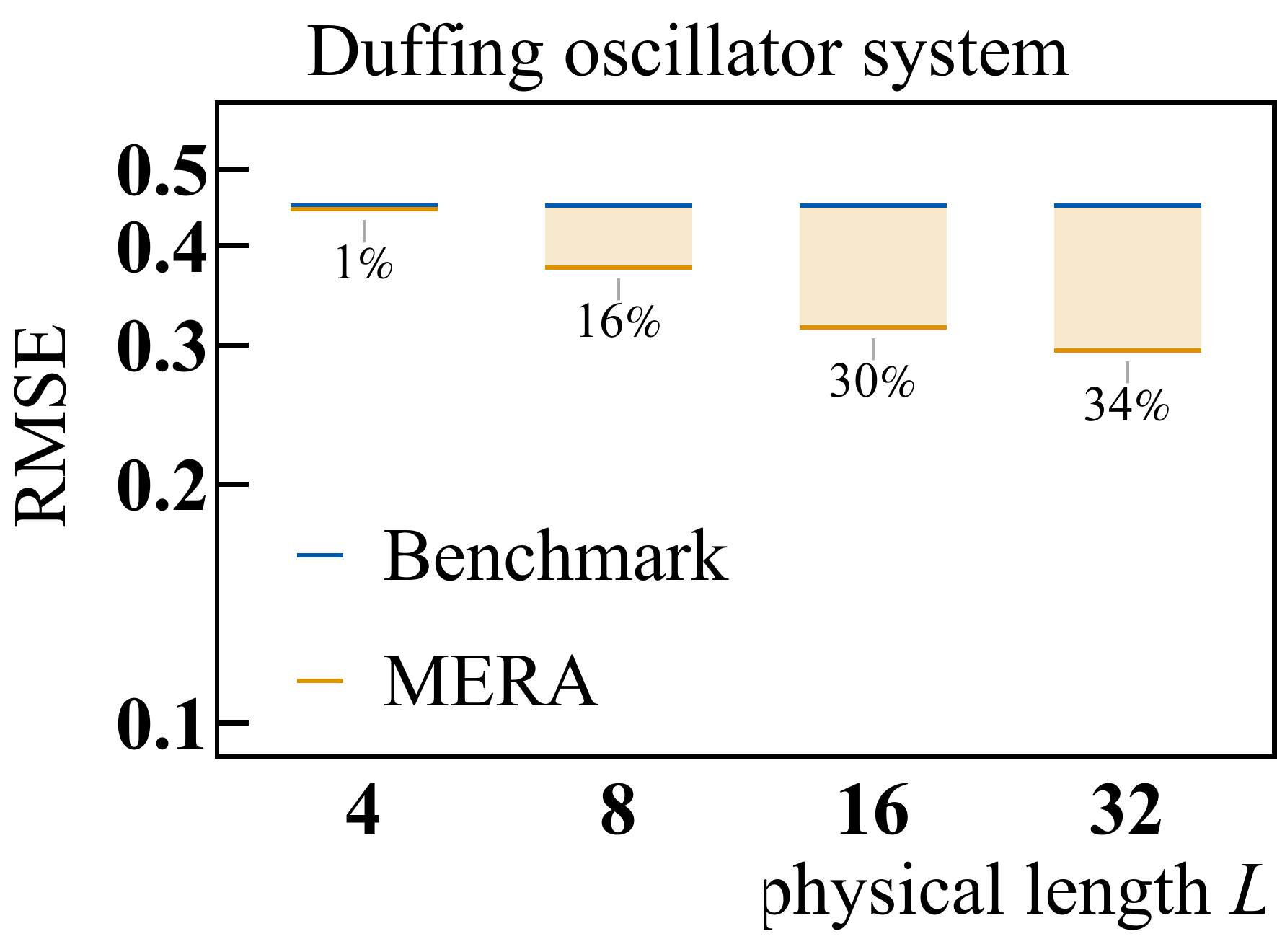}
			\vspace{-2mm}
			\subcaption{\label{fig_duffing}}
		\end{minipage}
		\addtocounter{subfigure}{1}
		\hspace{10mm}
		\begin{minipage}[t]{51mm}
			\includegraphics[width=50mm]{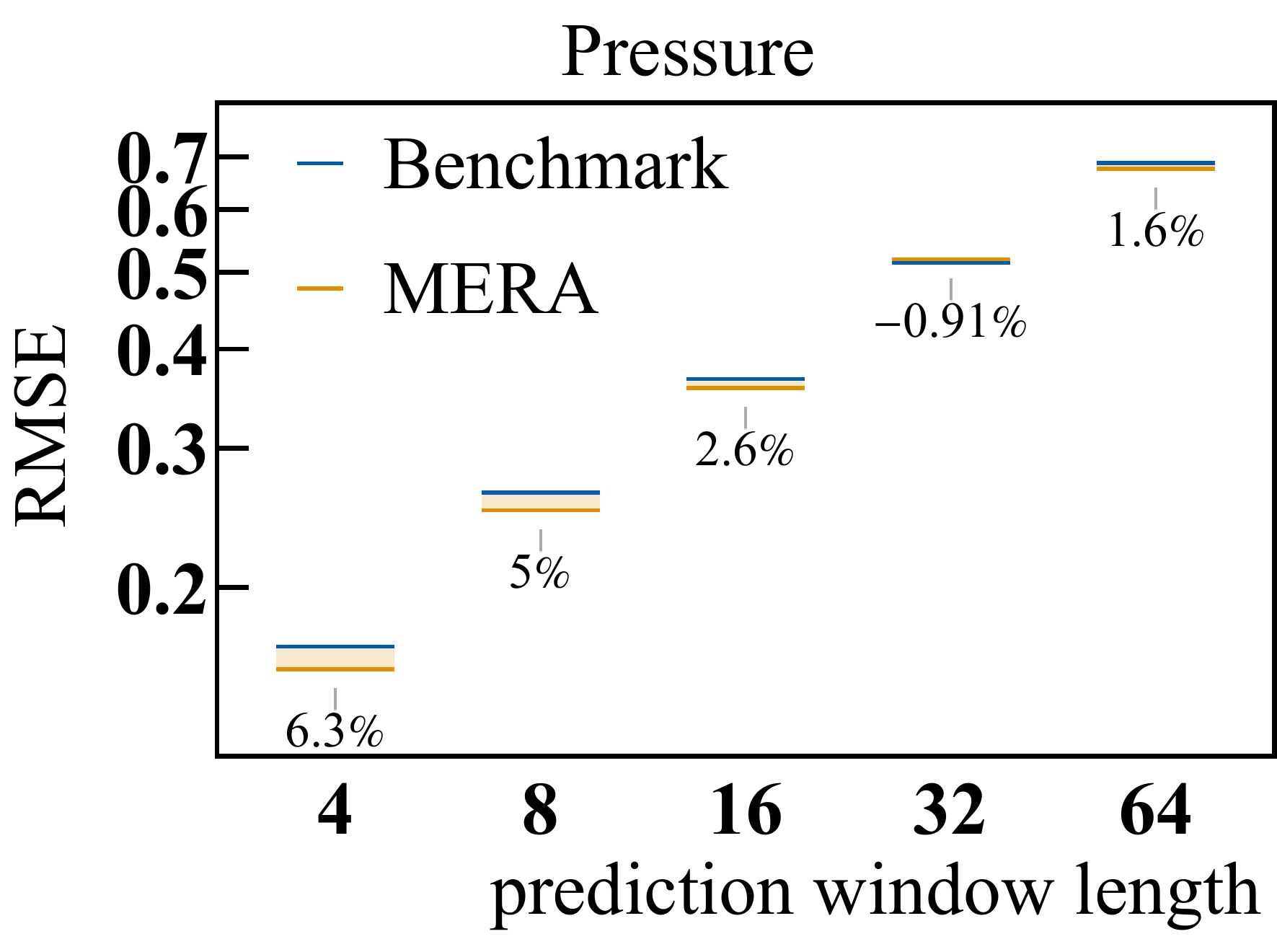}
			\vspace{-2mm}
			\subcaption{\label{fig_kabq}}
		\end{minipage}
		\addtocounter{subfigure}{-4}
		\begin{minipage}[t]{51mm}
			\includegraphics[width=50mm]{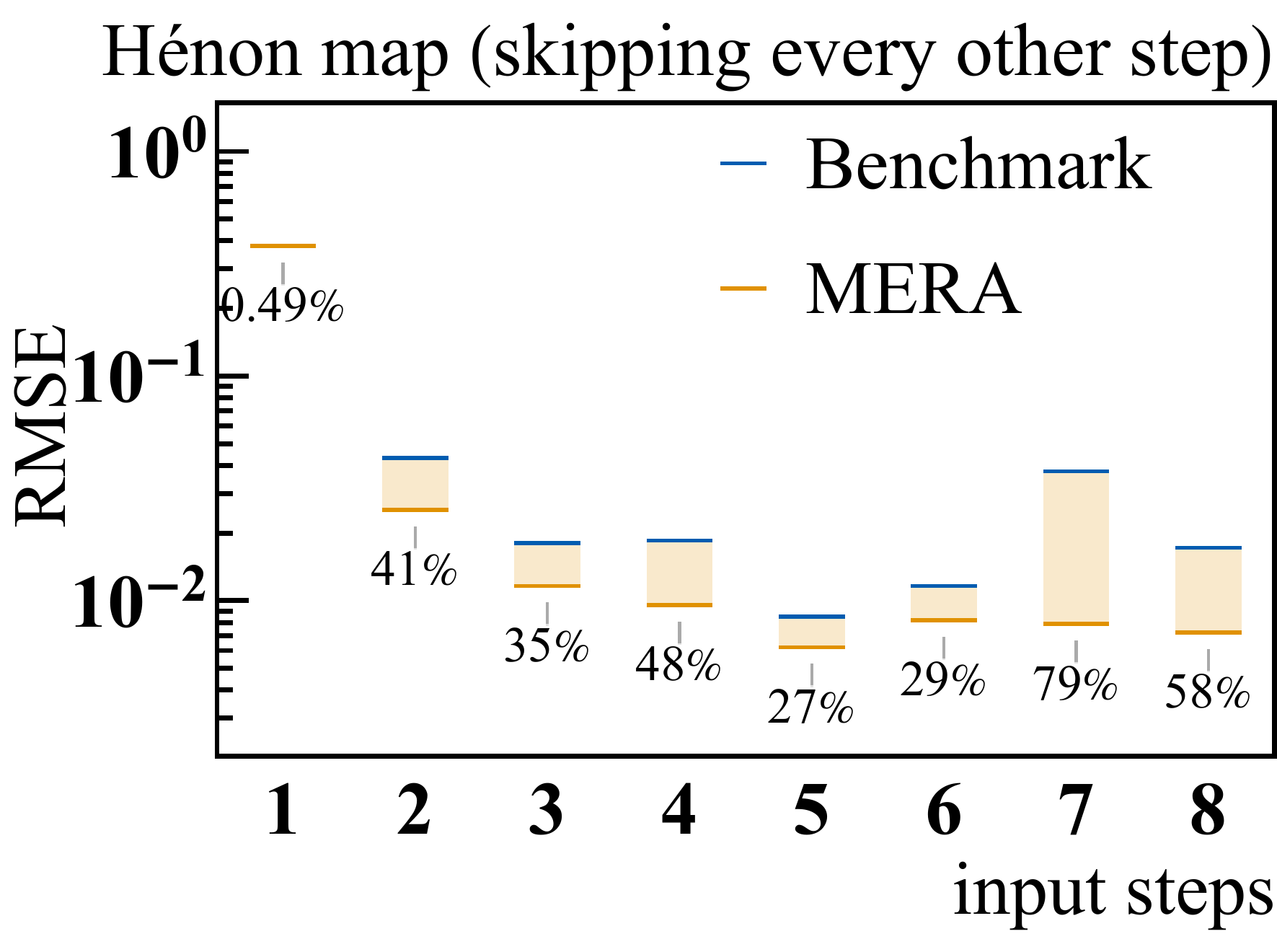}
			\vspace{-2mm}
			\subcaption{\label{fig_henon}}
		\end{minipage}
		\addtocounter{subfigure}{1}
		\begin{minipage}[t]{51mm}
			\includegraphics[width=50mm]{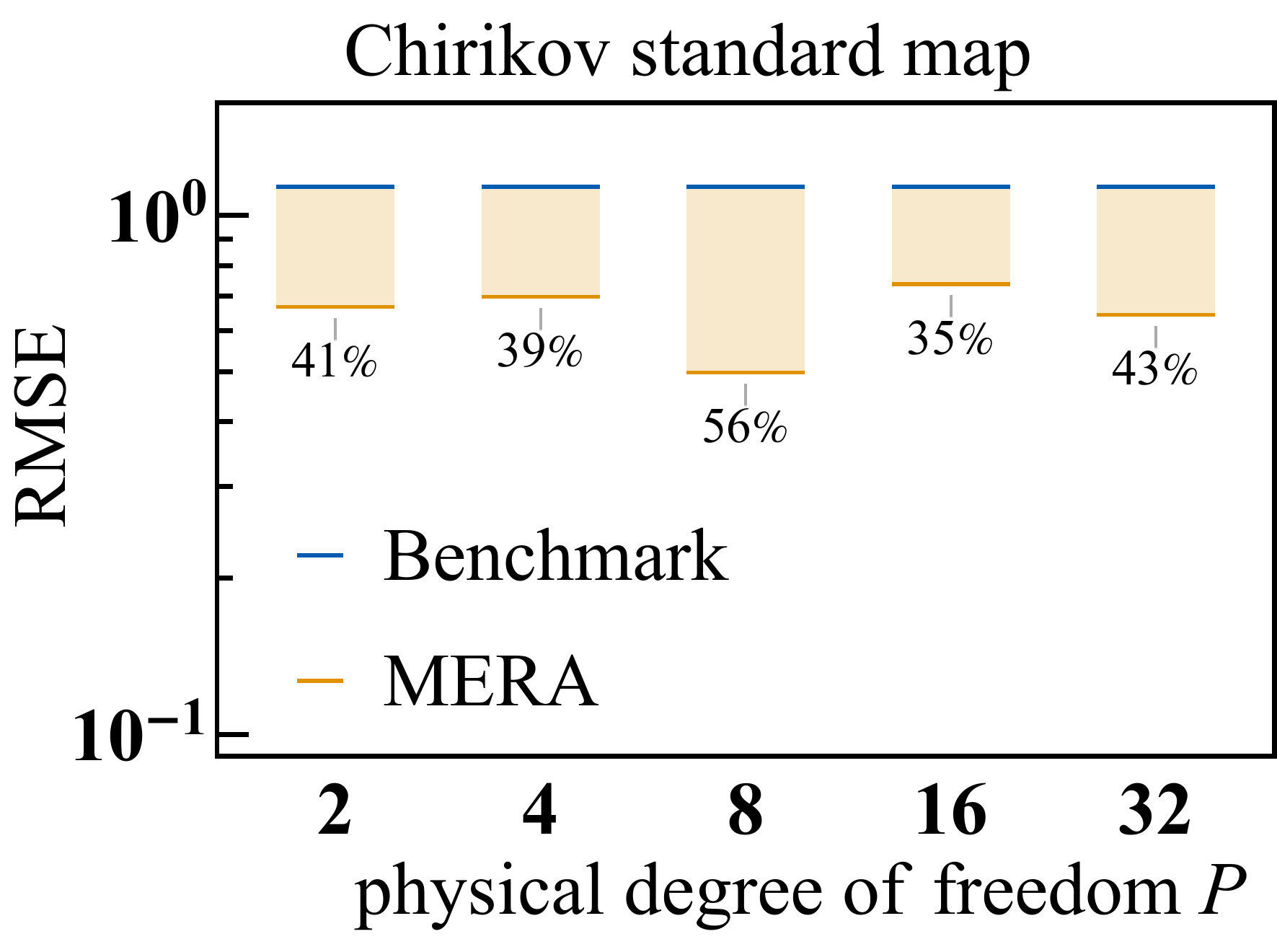}
			\vspace{-2mm}
			\subcaption{\label{fig_chirikov}}
		\end{minipage}
		\addtocounter{subfigure}{1}
		\hspace{10mm}
		\begin{minipage}[t]{51mm}
			\includegraphics[width=50mm]{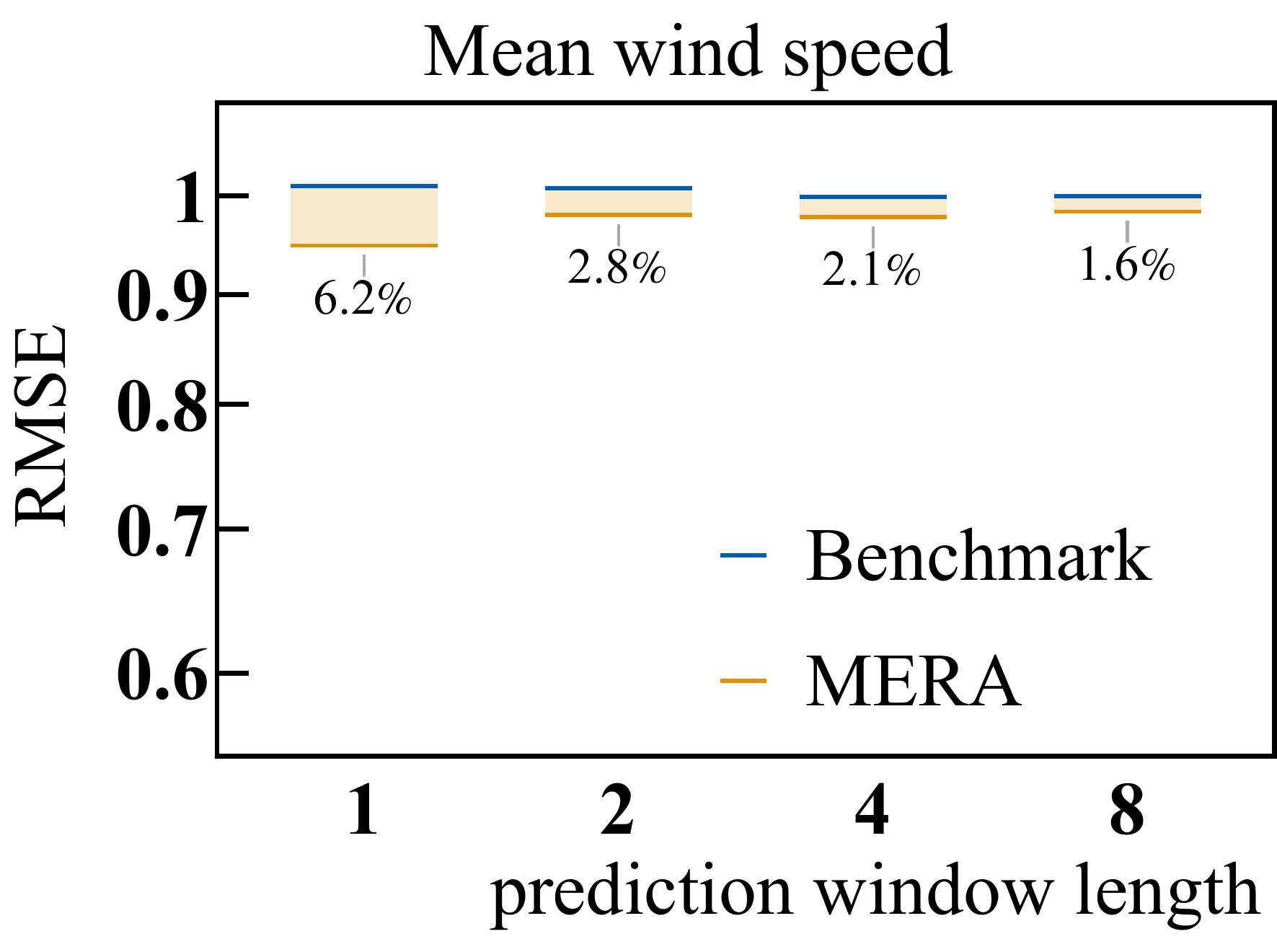}
			\vspace{-2mm}
			\subcaption{\label{fig_kbos}}
		\end{minipage}
		\vspace{-2mm}
		\caption{\label{fig_result1}Generalization and parameter dependence of LSTM-MERA.
			\subref{fig_rossler},~R\"ossler system, another three-dimensional chaotic dynamical system similar to the Lorenz system. Discretization: varying $\Delta t$. Input steps~$=4$, $h=4$, $L=2^4$, $P=2$, and $\{D_\text{I},D_\text{II},\cdots\}=\{P,2,2,4\}$.
			\subref{fig_henon},~One-dimensional, second-order H\'enon map, re-sampled by skipping every other step. $h=4$, $L=2^3$, $P=2$, and $\{D_\text{I},D_\text{II},\cdots\}=\{P,2,4\}$, while the input steps vary.		
			\subref{fig_duffing},~Duffing oscillator system. Discretization: $\Delta t=10.0$. Input steps~$=8$, $h=4$, $P=4$, and $\{D_\text{I},D_\text{II},\cdots\}=\{P,3,3,\cdots\}$ of which the length varies with $L$.
			\subref{fig_chirikov},~Chirikov standard map. Input steps~$=2$, $h=2$, $L=2^3$, and $\{D_\text{I},D_\text{II},\cdots\}=\{P,4,4\}$ where $P$ varies. 			
			\subref{fig_kabq},~Pressure, sampled every eight minutes. Input steps~$=16$, $h=4$, $L=2^4$, $P=4$, $\{D_\text{I},D_\text{II},\cdots\}=\{P,2,2,4\}$, and training : validation : text $=6400:1600$ $:\left({\gtrsim13800}\right)$.
			\subref{fig_kbos},~Mean wind speed, daily sampled. Input steps~$=16$, $h=4$, $L=2^3$, $P=4$, $\{D_\text{I},D_\text{II},\cdots\} =\{P,2,2\}$, and training : validation : text  $=128:32:\left({\gtrsim7000}\right)$.
		}
	\end{figure*}
	
	\subsection{Comparison with LSTM-MERA Alternatives}
	\label{sec_alternative}
	Here we tested the ability of LSTM-MERA for learning short-term nonlinear complexity by changing its NN topology (Fig.~\ref{fig_thomas}). We expected to see that, to achieve the best performance, our tensorization (dashed rectangle in Fig.~\ref{fig_netw}) should indeed act on the state propagation path $c_t\to s_t$, not on $s_{t-1}\to s_t$ or $c_{t-1}\to c_t$.
	
	\subsubsection{Thomas' cyclically symmetric system}
	We investigated different LSTM-MERA alternatives on \emph{Thomas' cyclically symmetric system} (Fig.~\ref{fig_thomas}) in order to see if the short-term complexity could still be efficiently learned. The embedded layers, besides located at Site~A (the proper NN topology of LSTM-MERA), were also located alternatively at Site~B,~C~or~D for comparison. The benchmark was a vanilla LSTM with no embedded layers.
	
	As expected, the lowest RMSE was produced by the proper LSTM-MERA but not its alternatives (Fig.~\ref{fig_thomas}). The improvement of the proper LSTM-MERA over the benchmark was $\sim 60\%$. Interestingly, two alternatives (Site~B, Site~C) performed \emph{barely} better than the benchmark even with more free learnable parameters.
	In fact, in the case that the state propagation path $c_{t-1}\to c_t$ is tensorized (Site~B), the long-term gradient propagation along cell states is interfered and the performance of LSTM deterred; when the path $s_{t-1}\to s_t$ is tensorized (Site~C), the improvement is the same as just on a plain RNN and thus also limited.
	Hence, proper LSTM-MERA NN topology is critical for improving the performance of learning short-term complexity.
	
	\subsection{Generalization and Parameter Dependence of LSTM-MERA}	
	The inherent advantage of LSTM-MERA and its ability to learn chaos have been shown. 
	Hereafter investigated are its parameter dependence as well as ability of generalization (Fig.~\ref{fig_result1}). Each following model (benchmark versus LSTM-MERA) was sufficiently trained through the same number of epochs so that it could reach the lowest stable RMSE. In-between check points were chosen during training where models were \emph{a posteriori} tested on the test data to confirm that an RMSE minimum had eventually been reached. 
	
	\subsubsection{R\"ossler system}
	In theory, a chaotic time series of larger $\Delta t$ should be harder to learn [Eq.~(\ref{eq_lyapunov-exp})]. This is confirmed in Fig.~\ref{fig_rossler} where larger $\Delta t$ corresponds to larger RMSE for both models. The most improvement of LSTM-MERA over the benchmark was $\sim 76\%$, observed at $\Delta t = 5$. The improvement was less when $\Delta t$ increased, possibly because the time series became too random to preserve any feasible pattern even for LSTM-MERA. The improvement was also less when $\Delta t$ was small, as the time series was smooth enough and the first-order (linear) time-dependence predominated which a vanilla LSTM could also learn.
	
	\subsubsection{H\'enon map}
	In view of the fact that the time-dependence is second-order [Fig.~\ref{fig_henon}], there was no explicit and exact dependence between the input and target in the time series dataset. Different input steps were chosen for comparison.
	When input steps~$=1$, there was no sufficient information to be learned other than a linear dependence between the input and target, and thus both the benchmark and LSTM-MERA performed the same [Fig.~\ref{fig_henon}]. When input steps~$>1$, however, the time-dependence could be learned implicitly and ``bidirectionally'' given enough history in length. LSTM-MERA constantly exhibited an average improvement of $~45.3\%$,  the fluctuation of which was mostly due to the learning instability of not LSTM-MERA but the benchmark.

	\subsubsection{Duffing oscillator system}
	From Fig.~\ref{fig_duffing} it was clearly observed that larger $L$ yielded better RMSE. The improvement related to $L$ was significant. This result is not unexpected, since $L$ determines the depth of the MERA structure, the larger which the better the ability of generalization should be.
	
	\subsubsection{Chirikov standard map}
	As Fig.~\ref{fig_chirikov} shows, by choosing different $P$, the most improvement of LSTM-MERA over the benchmark was $\sim56\%$, observed at $P=8$. In general, there was no strong dependence on $P$.
	
	\subsubsection{Real-world data: weather forecasting}
	The advantage of LSTM-MERA was also tested on real-world weather forecasting tasks [Figs.~\ref{fig_kabq}~and~\ref{fig_kbos}] .
	Unlike for synthetic time series, here we removed the first-layer translational symmetry [Eq.~(\ref{eq_trans})] previously imposed on LSTM-MERA so that presumed non-stationarity in real-world time series could be better addressed. To make practical multi-step forecasting, we kept the one-step-ahead prediction architecture of LSTM yet regrouped the original time series by choosing different prediction window lengths (Appendix~\ref{sec_data-process-real}).
	
	The improvement of LSTM-MERA over the benchmark was less significant. The average improvement was $\sim3.0\%$, while the most improvement was $\sim6.3\%$ given that the prediction window length was small, reflecting that LSTM-MERA is better at capturing short-term nonlinear complexity rather than long-term non-Markovianity. Note that, in the second dataset [Fig.~\ref{fig_kbos}], we deliberately used a very small number ($=128$) of training data to test the overfitting resistibility of the models. Interestingly, LSTM-MERA did not generally perform worse than vanilla LSTM even with more parameters, probably due to the deep architecture of LSTM-MERA. 
	
	\section{Discussion and Conclusion}
	Limitation of our model mostly comes from the fact that it is only better than traditional LSTM at capturing short-term nonlinearity but not long-term non-Markovianity, and thus its improvement on long-term tasks such as sequence prediction would be limited.
	That being said, the advantages of tensorizing state propagation in LSTM are evident, including: (1) Tensorization is the most suitable for forecasting of nonlinear chaos since nonlinear terms are treated explicitly and weighted equally by polynomials. (2) Theoretical analysis is conductible since an orthogonal polynomial basis on $k$-Sobolev space is always available. (3) Tensor decomposition techniques (in particular, from quantum physics) are applicable, which in turn can identify chaos from a different perspective, i.e., tensor complexity (tensor ranks, entropies, etc.).
	
	Our tensorized LSTM model not only offers a most general and efficient approach for capturing chaos---as demonstrated by both theoretical analysis and experimental results, showing more-than-ever potential in unraveling real-world time series, but also brings out a fundamental question of how tensor complexity is related to the learnability of chaos. Our conjecture that a tensor complexity of $S_{\alpha}(l)=\Theta(\log l)$ in terms of $\alpha$-R\'enyi entropy~[Eq.~(\ref{eq_ee})] generally performs better than $S_{\alpha}(l)=\Theta(1)$ at chaotic time series forecasting will be further investigated and formalized in the near future.

\nocite{*}
\bibliography{aipsamp}

\newpage
\appendix

	\section{Variants of LSTM-MERA}
	
	\subsection{Translation Symmetry}
	In condensed matter physics, the many-body states studied are usually translational invariant in $L$, which puts additional constraints on their many-body state structures (MPS, MERA, etc.). Inspired by this, a variant of LSTM-MERA can be constructed by imposing such constraints on the MERA structure too, i.e., by forcing the disentanglers belonging to the same level to be equal to each other. For example, on the first level of the MERA structure [Level I, red in Fig.~\ref{fig_notation}], a constraint
	\begin{eqnarray}
	\label{eq_trans}
	[u^{\left(i\right)}_{\text{I}}]^{\alpha\beta}_{\mu\nu}\equiv [u_{\text{I}}]^{\alpha\beta}_{\mu\nu},\quad i=1,2,\cdots,L/2
	\end{eqnarray}
	can be imposed on the weights of the $L/2$ disentanglers. Such a constraint can also be imposed on isometries/inverse isometries as well as higher levels. When testing LSTM-MERA on synthetic time series, we have added such a partial translational symmetry constraint on and only on Level I for the purpose of controlling the number of free learnable parameters in our model.
	
	\subsection{Dilational Symmetry}
	A dilational symmetry constraint is exclusive for MERA, since it has been used in condensed matter physics for representing scaling invariant quantum states. A variant of LSTM-MERA can thus be introduced by imposing the same constraint, i.e., by forcing all disentanglers even from different levels to be equal to each other,
	\begin{eqnarray}
	\label{eq_dilat}
	[u^{\left(i\right)}_{\text{I}}]^{\alpha\beta}_{\mu\nu}\equiv [u_{\text{I}}]^{\alpha\beta}_{\mu\nu}\equiv[u^{\left(j\right)}_{\text{II}}]^{\alpha\beta}_{\mu\nu}\equiv [u_{\text{II}}]^{\alpha\beta}_{\mu\nu}\equiv &\cdots&,\nonumber\\
	i=1,2,\cdots,L/2;\quad j=1,2,\cdots,L/4;\quad&\cdots&,
	\end{eqnarray}
	as well as isometries. This variant of LSTM-MERA may greatly decrease the number of free learnable parameters but may also lose the expressive power.

	\subsection{Normalization/Unitarity}
	Another subtle fact for many-body state structures is that the represented states must be normalized. In fact, normalization layers are also widely used in NN architectures, especially for deep NN where the training may suffer from the vanishing gradient problem. In light of this, we have added normalization layers between different LSTM-MERA layers. No extra freedom has been introduced because the ``norm'' is already a degree of freedom implicitly given by the weights of the disentanglers/isometries.
	
	Similarly, the unitarity of the disentanglers~\cite{mera_v08} is no longer required. The additional degrees of freedom do not affect the essential MERA structure but may significantly speed up our training.
	
	\section{Common LSTM Architectures}

	\subsection{HO- and HOT-RNN/LSTM}
	\label{sec_ho-hot}
	HO-RNN/LSTM~\cite{ho-rnn_sj16} were first introduced to address the problem of explicitly capturing long-term dependence, by changing all gates in LSTM [Eq.~(\ref{eq_lstm-operator})] into
	\begin{eqnarray}
	\label{eq_ho}
	g(x_{t-1},1\oplus s_{t-1}\oplus s_{t-2}\oplus \cdots\oplus s_{t-L};\mathcal{W}).
	\end{eqnarray}
	Since Eq.~(\ref{eq_ho}) only includes linear (first-order polynomial) terms, HOT-RNN/LSTM~\cite{hot-rnn_yzay19} were later introduced to include nonlinear (higher-order polynomial) terms as tensor products of the entire non-Markovian dependence,
	\begin{eqnarray}
	\label{eq_hot}
	g(x_{t-1},\left(1\oplus s_{t-1}\oplus s_{t-2}\oplus \cdots\oplus s_{t-L}\right)^{\otimes P};\mathcal{W}).
	\end{eqnarray}
	The weight tensor $\mathcal{W}$ can be further approximated by the tensor-train (i.e., MPS) technique~\cite{hot-rnn_yzay19}.
	
	Note that $L$ in Eqs.~(\ref{eq_ho})~and~(\ref{eq_hot}) is not a virtual dimension but the true time lag. Therefore, to increase the tensorization complexity one has to explicitly increase the time lag dependence. As a comparison, in LSTM-MERA, $L$ is an artificial dimension that can be freely adjusted to reflect the true short-term nonlinear complexity.
	
	\section{Preparation of Time Series Datasets}
	\label{sec_data-process}
	
	\subsection{Discrete-Time Maps}
	\label{sec_data-process-discrete}
	Each time series dataset for discrete-time maps was constructed as follows: first, two arrays were produced by the discrete-time map, one with initial conditions (training) and the other one with initial conditions (testing); next, for both arrays, a time window of fixed length ($\text{input steps }+1$) moved from the beginning to the end, step by step, and thus extracted a sub-array of length ($\text{input steps }+1$) at each step; each extracted sub-array was a time series. All time series (from both the training array and the testing array) made up the entire time series dataset and served for training and testing, respectively.
	The initial conditions for training and testing were made different \emph{on purpose} in order to test the generalization ability of the models, yet they were chosen to belong to the same chaotic regime so that the generality of their subsequent chaotic dynamics was always guaranteed by ergodicity. We investigated 4 different dynamical systems:
	\begin{align}
	\text{\textbf{Logistic map:}} &\qquad x_{n+1} = rx_n(1-x_n); \nonumber \\
	\text{\textbf{Gauss iterated map:}} &\qquad x_{n+1}=\exp\left(-\alpha x_{n}^2\right)+\beta; \nonumber \\
	\text{\textbf{H\'enon map:}} &\qquad x_{n+1}=1-a x^2_{n}+b x_{n-1}; \nonumber \\
	\text{\textbf{Chirikov standard map:}} &\qquad p_{n+1}=\left(p_{n}+ K \sin\theta_n\right)\text{ mod }2\pi, \nonumber \\
	&\qquad \theta_{n+1}=\left(\theta_{n}+p_{n+1}\right)\text{ mod }2\pi.
	\nonumber
	\end{align}
	Details of above systems are listed in Table-\ref{discrete_details}.

	\begin{table}[!ht]
		\begin{center}
			\begin{tabular}{|c|c|c|c|c|}
				\hline
				& Logistic & Gauss & H\'enon & Chirikov \\
				\hline
				Dimension & 1 & 1 & 1 & 2 \\
				\hline
				Parameters & $r=4$ & $\alpha=6.2$ & $a=1.4$ & $K=2.0$ \\
				&  & $\beta=-0.55$ & $b=0.3$ & \\
				\hline
				Initial condition & $x_0=0.61$ & $x_0=0.31$ & $x_0=0.2$ & $p_0=0.777$ \\
				(training) &  &  & $x_1=0.3$ & $\theta_0=0.555$ \\
				\hline
				Initial condition & $x_0=0.11$ & $x_0=0.91$ & $x_0=0.5$ & $p_0=0.333$ \\
				(testing) &  &  & $x_1=0.6$ & $\theta_0=0.999$ \\
				\hline
				$\lambda_1$ & $\ln{2}$ & $0.37$ & $0.42$ & $0.45$ \\
				$\lambda_2$ & -- & -- & $-1.62$ & $-0.45$ \\
				\hline
			\end{tabular}
		\end{center}
		\caption{Implementation details of 4 continuous dynamical systems in chaotic phases. $\lambda_{1,2}$ are Lyapunov exponents.}
		\label{discrete_details}
	\end{table}

	\subsection{Continuous-Time Dynamical Systems}
	\label{sec_data-process-continuous}
	
	Each time series dataset for continuous-time dynamical systems was constructed differently than in Section~\ref{sec_data-process-discrete}: only one array was produced by discretizing the dynamical system by $\Delta t$ given the initial conditions; then the array was \emph{standardized}; a time window still moved from the beginning to the end and extracted a sub-array of length ($\text{input steps }+1$) at each step; each extracted sub-array was a time series. All time series made up the entire time series dataset which was then randomly divided into two subsets, one for testing and one for training. Four different dynamics are investigated:
	\begin{align*}
	&\text{\textbf{Lorentz system:}} 
	\nonumber \\
	& \qquad 
	\frac{dx}{dt}=\sigma\left(y-x\right),
	\;\frac{dy}{dt}=x\left(\rho-z\right)-y,
	\;\frac{dz}{dt}=xy-\beta z;
	\nonumber\\
	&\text{\textbf{Thomas' cyclically symmetric system:}}
	\nonumber \\
	& \qquad 
	\frac{dx}{dt}=\sin y -bx,
	\;\frac{dy}{dt}=\sin z -by,
	\;\quad\frac{dz}{dt}=\sin x -bz;
	\nonumber \\
	&\text{\textbf{R\"ossler system:}}
	\nonumber \\
	& \qquad 
	\frac{dx}{dt}=-y-z,
	\;\quad\frac{dy}{dt}=x+ay,
	\;\;\;\qquad\frac{dz}{dt}=b+z\left(x-c\right);
	\nonumber\\
	&\text{\textbf{Duffing oscillator system:}}
	\nonumber \\
	& \qquad 
	\frac{d^2x}{{dt}^2}+\delta\frac{dx}{dt}+\alpha x+\beta x^3 =\gamma \cos\left(\omega t\right).
	\nonumber
	\end{align*}
	Details of above systems are listed in Table-\ref{continuous_details}.
	
	\begin{table}
		\begin{center}
			\begin{tabular}{|c|c|c|c|c|}
				\hline
				& Lorentz & Thomas & R\"ossler & Duffing \\
				\hline
				Dimension & 3 & 3 & 3 & 1 \\
				\hline
				&  &  &  & $\alpha=1.0$ \\
				& $\rho=28$ &  & $a=0.1$ & $\beta=5.0$ \\
				Parameters & $\sigma=10.0$ & $b=0.1$ & $b=0.1$ & $\delta=0.02$ \\
				& $\beta=8/3$ &  & $c=14$ & $\gamma=8.0$\\ 
				&  &  &  & $\omega=0.5$ \\
				\hline
				& $x_0=0$ & $x_0=0$ & $x_0=0$ & $x_0=0$ \\
				Initial condition & $y_0=1$ & $y_0=1$ & $y_0=1$ & $\dot{x}_0=1$ \\
				& $z_0=0$ & $z_0=0$ & $z_0=0$ &  \\
				\hline
				$\lambda_1$ & $0.91$ & $0.06$ & $0.07$ & $0.01$\\
				$\lambda_2$ & $0$ & $0$ & $0$ & $0$ \\
				$\lambda_3$ & $-14.57$ & $-0.36$ & $-11.7$ & $-0.03$ \\
				\hline
				$T_{\max}$ & $[0, 2500]$ & $[0, 5000]$ & $[0, 100000]$ & $[0, 50000]$ \\
				\hline
			\end{tabular}
		\end{center}
		\vspace{-0.1cm}
		\caption{Implementation details of 4 continuous dynamical systems in chaotic phases. $T_{\max}$ is the maximum solution range, and $\lambda_{1,2,3}$ are Lyapunov exponents.}
		\label{continuous_details}
	\end{table}
		
	\subsection{Real-World Time Series: Weather}
	\label{sec_data-process-real}
	The data were retrieved by Mathematica's \emph{WeatherData} function\footnote{\url{https://reference.wolfram.com/language/note/WeatherDataSourceInformation.html}} (Fig.~\ref{fig_real}). And detailed information about the data has been provided in Table-\ref{data_info} Missing data points in the raw time series were reconstructed by linear interpolation. The raw time series was then regrouped by choosing different \emph{prediction window} length: for example, prediction window length~$=4$ means that every four consecutive steps in the time series are regrouped together as a one-step four-dimensional vector. Then, the dataset was constructed from the regrouped time series the same way as in Section~\ref{sec_data-process-continuous} by a moving window on it after standardization.
	\newpage
	\begin{figure}
		\centering 
		{\includegraphics[scale=0.4]{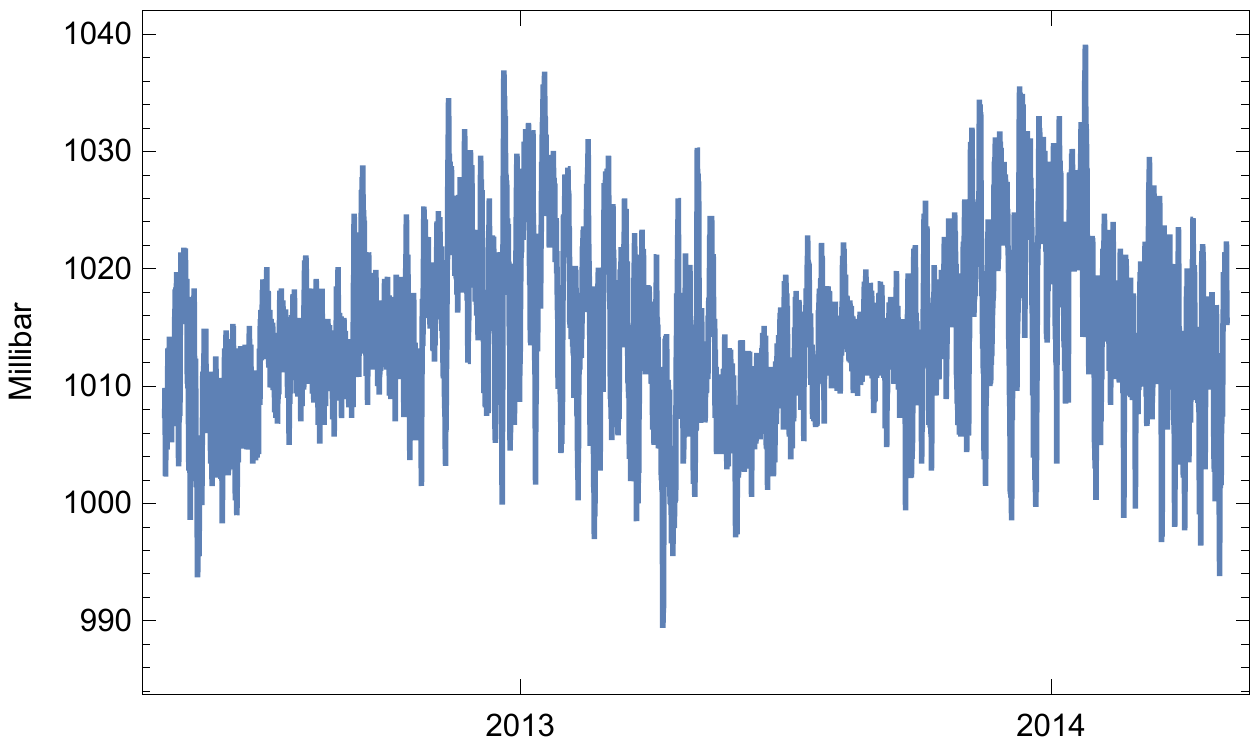}
			\subcaption{\label{fig_pressure}}}
		{\includegraphics[scale=0.4]{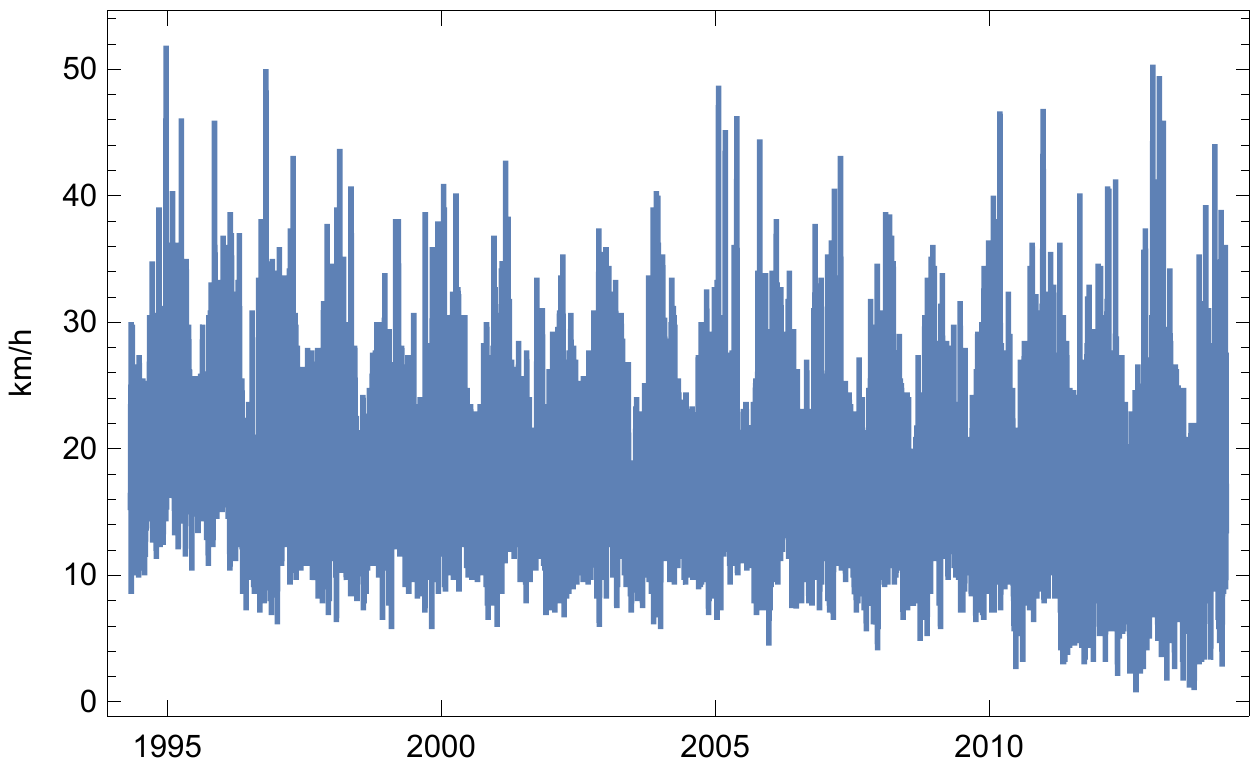}
			\subcaption{\label{fig_wind}}}
		\vspace{-0.3cm}
		\caption{\label{fig_real}~Weather time series. \subref{fig_pressure}~Pressure (every eight minutes) before standardization. \subref{fig_wind}~Mean wind speed (daily) before standardization.}
	\end{figure}
	\begin{table}
		\centering
		\begin{tabular}{|c|c|c|}
			\hline
			& Pressure & Mean wind speed \\
			\hline
			Location & ICAO:KABQ & ICAO:KBOS \\
			\hline
			Span & 05/01/2012 - 05/01/2014 & 05/01/1994 - 05/01/2014\\
			\hline
			Frequency & 8 min & 1 day\\
			\hline 
			Total length & 22426 & 7299\\
			\hline
		\end{tabular}
		\vspace{-0.2cm}
		\caption{Information details of weather datasets used in the main article.}
		\label{data_info}
	\end{table}

\end{document}